\documentclass[reqno]{amsart}

\usepackage[english]{babel}
\usepackage{amsmath,amssymb,amsthm}
\usepackage{hyperref}
\usepackage[centering]{geometry}
\usepackage{xcolor}

\renewcommand{\epsilon}{\varepsilon}            

\newtheorem{theorem}{Theorem}[section]   
\newtheorem*{theorem*}{Theorem}          
\newtheorem{lemma}[theorem]{Lemma}
\newtheorem{proposition}[theorem]{Proposition}

\theoremstyle{definition}

\newtheorem{remark}{Remark}[section]

\numberwithin{equation}{section}

\title[Hypersurfaces of $\mathbb{S}^2\times\mathbb{S}^2$ with constant sectional curvature]
{Hypersurfaces of $\mathbb{S}^2\times\mathbb{S}^2$ with constant sectional curvature}
\thanks{}

\author{Haizhong Li, Luc Vrancken, Xianfeng Wang \and Zeke Yao}


\date{}
\begin{document}

\begin{abstract}
In this paper, we classify  the hypersurfaces of $\mathbb{S}^2\times\mathbb{S}^2$
with constant sectional curvature.  By applying the so-called Tsinghua principle,
which was first discovered by the first three authors in 2013 at Tsinghua University,
we prove that the constant sectional curvature can only be $\frac{1}{2}$
and the product angle function $C$ defined by Urbano is identically zero.
We show that any such hypersurface is a parallel hypersurface of a minimal hypersurface in $\mathbb{S}^2\times\mathbb{S}^2$ with $C=0$, and we establish  a one-to-one correspondence between the involving minimal hypersurface and the famous
``sinh-Gordon equation''
$$
(\frac{\partial^2}{\partial u^2}+\frac{\partial^2}{\partial v^2})h
=-\tfrac{1}{\sqrt{2}}\sinh(\sqrt{2}h).
$$
As a byproduct, we give a complete classification of the hypersurfaces
of $\mathbb{S}^2\times\mathbb{S}^2$ with constant mean curvature and constant product angle function $C$.

\end{abstract}

\maketitle

\section{Introduction}\label{sect:1}

Besides the four-dimensional unit sphere $\mathbb{S}^4$ and the complex projective
plane $\mathbb{C}P^2$ of constant holomorphic sectional curvature $4$, the Riemannian
product $\mathbb{S}^2\times\mathbb{S}^2$ of two unit spheres is the most interesting
compact $4$-manifold. It is well known that  $\mathbb{S}^2\times\mathbb{S}^2$ and $\mathbb{C}P^2$ are the only compact Hermitian symmetric $4$-manifolds.
In recent years, the study of the canonical submanifolds of
$\mathbb{S}^2\times\mathbb{S}^2$ have become very active.  When the comdimension is  $2$, there have been many interesting results on minimal surfaces in $\mathbb{S}^2\times\mathbb{S}^2$ (cf. \cite{C-U,LMVVW,T-U2}, etc.) and surfaces with parallel mean curvature vector in $\mathbb{S}^2\times\mathbb{S}^2$  (cf. \cite{T-U1}, etc.). In hypersurface case (i.e., codimension $1$), it is remarable that  Urbano \cite{Ur}  classified the homogenous hypersurfaces and the isoparametric hypersurfaces of $\mathbb{S}^2\times\mathbb{S}^2$, and he also studied
the hypersufaces of $\mathbb{S}^2\times\mathbb{S}^2$ with constant principal curvatures and gave some other important classification results.
Recently, Gao, Hu, Ma and the fourth author \cite{G-H-M-Y} studied the  Hopf  hypersurfaces of $\mathbb{S}^2\times\mathbb{S}^2$ and the  hypersurfaces with isometric Reeb flow. Lu, Wang and Wang \cite{L-W-W} constructed some new examples of cohomogeneity one hypersurfaces in $\mathbb{S}^2\times\mathbb{S}^2$  with constant $C$. Nevertheless, many fundamental problems remain open.


From the point of view of Riemannian geometry, the classification problem of submanifolds with constant sectional curvature in different ambient spaces is an important and  attractive topic in differential geometry, and there have been lots of intereting
results in this respect. When the ambient space is a real space form, the classification of hypersurfaces with constant sectional curvature has a long history and have been well understood.
In non-flat complex space forms, there do not exist hypersurfaces with constant sectional curvature
cf. \cite{CR1982,I-R,K-R,Kon,M}, etc.  When the curvature tensor of ambient space does not have a simple expression,
the classification problem of hypersurfaces with
constant sectional curvature becomes very difficult, if one uses the Gauss and Codazzi
equations directly. To overcome this difficulty, the first three authors introduced a new approach,
the so-called {\it Tsinghua principle}, which applies the Codazzi equation
and the Ricci identity in a indirect way to obtain some nice linear equations involving
the components of the second fundamental form. By using such approach,
some canonical submanifolds with constant sectional curvature (even under more general conditions)
in some canonical Riemannian manifold have been classified,
cf. \cite{A-L-V-W,CHMV-1,D-V-W,LMVVW,V-W,Y-Y-H}, etc.
The main purpose of this paper is to classify the hypersurfaces of $\mathbb{S}^2\times \mathbb{S}^2$
with constant sectional curvature. 

Before stating our main result, we first recall some basis properties and some canonical examples. Note that there is a natural product
structure $P$ on $\mathbb{S}^2\times\mathbb{S}^2$ defined by
$$
P(X_1,X_2)=(X_1,-X_2)
$$
for any tangent vector field $X_1, X_2$ on $\mathbb{S}^2$. Then, as observed by Urbano
\cite{Ur}, the geometry of the hypersurfaces in $\mathbb{S}^2\times\mathbb{S}^2$
is closely related to a smooth function $C$. Here, for an orientable hypersurface $M$
of $\mathbb{S}^2\times\mathbb{S}^2$ with $N$ a unit normal vector field, the function
$C$ is defined by $$C=g(PN,N),$$ where $g$
denotes the standard metric on $\mathbb{S}^2\times\mathbb{S}^2$. Hereafter, we  call $C$ the {\it product angle function} of $M$.
It is clear that
$-1\le C\le1$. Urbano \cite{Ur} proved that any hypersurface of $\mathbb{S}^2\times\mathbb{S}^2$ with $C^2\equiv 1$ is locally the product of a curve in $\mathbb{S}^2$ and an open subset of $\mathbb{S}^2$.
We denote the standard metric on $\mathbb{S}^2$ by $\langle\cdot,\cdot\rangle $.
Then, as typical examples, for $t\in(-1,1)$ and unit vectors $a,b\in\mathbb{S}^2$,
Urbano \cite{Ur} showed that the following three kinds of hypersurfaces
in $\mathbb{S}^2\times\mathbb{S}^2$ have constant $C=0$ and vanishing Gauss-Krocneker curvature:
\begin{align*}
&M_t=\{(p, q)\in\mathbb{S}^2\times\mathbb{S}^2\,|\,\langle p,q\rangle=t\};\\
&M_{a, b}=\{(p, q)\in\mathbb{S}^2\times\mathbb{S}^2\,|\,\langle p,a\rangle+\langle q,b\rangle=0\};\\
&\hat{M}_{a, b}=\{(p,q)\in\mathbb{S}^2\times\mathbb{S}^2\,|\,
\langle p,a\rangle^2+\langle q,b\rangle^2=1\}.
\end{align*}
Moreover, the above examples have the following properties \cite{Ur} :

\begin{itemize}
	\item [(i)]
$M_t$ is a homogeneous isoparametic hypersurface with three constant principal curvatures (hence with constant mean curvature). Among all the {$M_t$}, only $M_0$ is minimal.
\item [(ii)] $M_{a, b}$ is minimal with non-constant scalar curvature.
\item [(iii)] $\hat{M}_{a, b}$ has constant sectional curvature $1/2$ (hence with constant scalar curvature) with non-constant mean curvature.
\end{itemize}

As the first important result in this paper, we classify  the hypersurfaces of $\mathbb{S}^2\times\mathbb{S}^2$
with constant sectional curvature.

\begin{theorem}\label{thm:1.1}
Let $M$ be a hypersurface of $\mathbb{S}^2\times \mathbb{S}^2$ with constant sectional curvature $\kappa$.
Then $\kappa=1/2$, the product angle function $C=0$ and  $M$  is congruent to an open part of a parallel hypersurface of a minimal hypersurface in $\mathbb{S}^2\times\mathbb{S}^2$ with $C=0$. More precisely, $M$ is locally either
\begin{enumerate}
\item[(1)]
an open part of $\hat{M}_{a,b}$ for some $a,b\in\mathbb{S}^2$; or

\item[(2)]
an open part of a parallel hypersurface of $\tilde{M}$ at distance $\frac{\pi}{2\sqrt{2}}$,
where $\tilde{M}$ is a minimal hypersurface in $\mathbb{S}^2\times\mathbb{S}^2$ with
$C=0$ described by Theorem \ref{thm:3.1}.
\end{enumerate}
\end{theorem}

\begin{remark}\label{rem:1.1}
(i) In Theorem	\ref{thm:3.1}, we establish  a one-to-one correspondence between the involving minimal hypersurface $\tilde{M}$ in Case (2) of  Theorem \ref{thm:1.1} and the famous
``sinh-Gordon equation''
$$
(\frac{\partial^2}{\partial u^2}+\frac{\partial^2}{\partial v^2})h
=-\tfrac{1}{\sqrt{2}}\sinh(\sqrt{2}h).
$$
In fact, we prove that $\tilde{M}$ admits coordinates $(u,v,t)$ and there is a non-zero function $h:(u,v)\to\mathbb{R}$ such that $h$ satisfies the above ``sinh-Gordon equation''. Conversely,
 for any non-zero solution $h:(u,v)\rightarrow \mathbb{R}$, we construct a metric $g$ and two
$(1,1)$-tensors $A,~\tilde{P}$ (see \eqref{eq:2.ab17},  \eqref{eqn:A} and \eqref{defP}), then by the existence theorem
and uniqueness theorems in \cite{K,L-T-V}, up to an isometry of $\mathbb{S}^2\times\mathbb{S}^2$, there exists a unique minimal hypersurface in $\mathbb{S}^2\times\mathbb{S}^2$
such that $C=0$ and its first fundamental form and shape operator  are given by $g$ and $A$,
respectively.
(ii) The hypersurface $\hat{M}_{a,b}$ in Case (1) of  Theorem \ref{thm:1.1},   is actually a parallel hypersurface of  the minimal hypersurface $M_{a,b}=\{(p,q)\in
\mathbb{S}^2\times \mathbb{S}^2|\ \langle p,a\rangle+\langle q,b\rangle=0\}$.
Note that  $M_{a,b}$ corresponds to a trivial solution of ``sinh-Gordon equation'': $h\equiv0$.
\end{remark}

As the second main result, we obtain a complete classification of
the hypersurfaces of $\mathbb{S}^2\times\mathbb{S}^2$
with constant $C$ and constant mean curvature (or constant scalar curvature),
which completes the the classification result by Urabno (see Theorem 1 in  \cite{Ur}).

\begin{theorem}\label{thm:1.2}
Let $M$ be a hypersurface of $\mathbb{S}^2\times \mathbb{S}^2$ with constant product angle function $C$. Then
\begin{enumerate}
\item[(1)]
$M$ has constant mean curvature if and only if either
\begin{enumerate}
\item[(a)]
$C^2=1$ and $M$ is an open part of $\mathbb{S}^1(r)\times \mathbb{S}^2$ for some $r\in(0,1]$, or

\item[(b)]
 $C=0$ and $M$ is an open part of $M_t$ for some $t\in(-1,1)$, or

\item[(c)]
$C=0$ and $M$ is an open part of $M_{a,b}$ for some $a,b\in\mathbb{S}^2$,
or

\item[(d)]
$C=0$ and $M$ is an open part of a minimal hypersurface described by Theorem \ref{thm:3.1}.
\end{enumerate}

\item[(2)]
$M$ has constant scalar curvature if and only if either

\begin{enumerate}
\item[(a)]
$C^2=1$ and $M$ is an open part of $\mathbb{S}^1(r)\times \mathbb{S}^2$ for some $r\in(0,1]$, or

\item[(b)]
$C=0$ and $M$ is an open part of $M_t$ for some $t\in(-1,1)$, or

\item[(c)]
$C=0$ and $M$ is an open part of $\hat{M}_{a,b}$ for some $a,b\in\mathbb{S}^2$, or

\item[(d)]
$C=0$ and $M$ is an open part of a parallel hypersurface of $\tilde{M}$ at distance $\frac{\pi}{2\sqrt{2}}$, where $\tilde{M}$ is a minimal hypersurface of
$\mathbb{S}^2\times\mathbb{S}^2$ with $C=0$ described by Theorem \ref{thm:3.1}.
\end{enumerate}

\end{enumerate}
\end{theorem}

\begin{remark}\label{rem:1.2}
Our contributions in Theorem \ref{thm:1.2}  are Case (1d) and Case (2d).
Case (1a)--Case(1c) and  Case (2a)--Case(2c) were obtained by Urbano
(see   Theorem 1 in  \cite{Ur}).
	\end{remark}

The paper is organized as follows. In Section \ref{sect:2}, we collect some basic
properties of $\mathbb{S}^2\times\mathbb{S}^2$ and some preliminaries of the geometry of the hypersurfaces of $\mathbb{S}^2\times\mathbb{S}^2$.
In Section \ref{sect:3}, we give a complete classification of minimal hypersurfaces of $\mathbb{S}^2\times\mathbb{S}^2$
with $C=0$ (see Theorem \ref{thm:3.2aa}). It turns out that the resulting hypersurfaces can only be $M_0$,  $M_{a,b}$, or  one of the hypersurfaces described by Theorem \ref{thm:3.1}.  In order to obtain the classification, we discuss all the possibilities of components of the second fundamental form.  It is worth mentioning that we construct new examples of hypersurfaces with constant $C$ and give local characterizations of them. In fact, these examples are all the  hypersurfaces of $\mathbb{S}^2\times\mathbb{S}^2$
with constant $C\neq\pm 1$ and $b_2=0$ (see \eqref{eqn:2.7} for the definition of $b_2$). See Proposition \ref{prop:3.61}
for more details. In Sections \ref{sect:4}, we present the proofs of
Theorem \ref{thm:1.1} and Theorem \ref{thm:1.2}. We first apply the so-called Tsinghua principle
to prove that the constant sectional curvature can only be $\frac{1}{2}$ and
the product angle function $C$ is identically zero. This is one of key steps in the proof of Theorem \ref{thm:1.1}. Another important ingredient in the proof is an observation that any hypersurface with constant sectional curvature is a parallel hypersurface of a minimal hypersurface with $C=0$ (see Theorem \ref{thm:4.6}). Thus we can apply the classification result in Section
 \ref{sect:3} to complete the proofs.

\textbf{Acknowledgments:}
H. Li  was supported by NSFC Grant No. 11831005, NSFC Grant No. 12126405
and NSFC-FWO Grant No. 11961131001. L. Vrancken was supported by NSFC-FWO 11961131001. 
X. Wang was supported by NSFC Grant No. 11971244 and the Fundamental Research Funds
for the Central Universities. Z. Yao was supported
by NSFC Grant No. 12171437 and China Postdoctoral Science Foundation
(No.2022M721871).

\section{Preliminaries} \label{sect:2}

\subsection{The geometric structure on $\mathbb{S}^2\times\mathbb{S}^2$}\label{sect:2.1}~

Let $\mathbb{S}^{2}$ be the $2$-dimensional unit sphere with standard metric
$\langle\cdot,\cdot\rangle $ and complex structure $J$ defined by
$$
J_pv=p\wedge v,
$$
for any $p\in\mathbb{S}^2$ and $v\in T_p\mathbb{S}^2$, where $\wedge$ stands
for the cross product in $\mathbb{R}^3$. On the product space $\mathbb{S}^2\times\mathbb{S}^2$
with the product metric denoted by $g$, we
have two complex structures
$$
J_{1}=(J, J), \quad J_{2}=(J,-J),
$$
which define two K\"ahler structures on $\mathbb{S}^2\times\mathbb{S}^2$, so that
$\mathbb{S}^2\times\mathbb{S}^2$ becomes a K\"ahler surface.

The product structure $P$ on $\mathbb{S}^2\times\mathbb{S}^2$ is defined by
$P: T(\mathbb{S}^2\times\mathbb{S}^2)\rightarrow T(\mathbb{S}^2\times\mathbb{S}^2)$ such that
$$
P(X_1,X_2)=(X_1,-X_2), \quad \forall\, X_1, X_2\in T\mathbb{S}^2.
$$
$P$ satisfies that  $P=-J_1J_2=-J_2J_1$, $P^2=\mathrm{Id}$ and  $\bar{\nabla} P=0$, where $\bar{\nabla}$ is the Levi-Civita connection on
$\mathbb{S}^2\times\mathbb{S}^2$. Moreover,
$$
g(PY, Z)=g(Y,PZ), \quad \forall\,Y,Z\in T(\mathbb{S}^2\times\mathbb{S}^2).
$$

The Riemannian curvature tensor $\bar{R}$ of $\mathbb{S}^2\times\mathbb{S}^2$ with the product
metric is given by
\begin{equation*}
\begin{aligned}
\bar{R}(U,Y,Z,W)=\tfrac{1}{2}&\big\{g(Y,Z)g( U,W)-g( U,Z)g( Y,W)\\
&+g( PY,Z)g( PU,W)-g( PU,Z)g( PY,W)\big\},
\end{aligned}
\end{equation*}
where $U, Y, Z, W\in T(\mathbb{S}^2\times\mathbb{S}^2)$. Thus, $\mathbb{S}^2\times\mathbb{S}^2$
is an Einstein manifold with scalar curvature $4$ and nonnegative sectional curvature.

\subsection{Hypersurfaces of $\mathbb{S}^2\times\mathbb{S}^2$}\label{sect:2.2}~

Let $M$ be an orientable hypersurface of $\mathbb{S}^2\times\mathbb{S}^2$ with
$N$ a unit normal vector field and still $g$ the induced metric on $M$. Then, with respect
to the product structure $P$, we define by
\begin{align*}
C&:=g(PN,N)=g(J_1 N,J_2 N),\\
X&:=PN-CN,
\end{align*}
the product angle function $C: M\rightarrow\mathbb{R}$ and a vector field
$X$ tangent to $M$. It is clear that $-1\leq C\leq 1$ and $|X|^2:=g(X,X)=1-C^2$.
For any tangential vector field $Y$ of $M$, acting by the product structure $P$,
we have the following decomposition
$$
PY=T Y+\mu(Y)N,
$$
where $T Y$ and $\mu(Y)N$ are the tangential
and normal part of $PY$. Thus $T$ is a
tensorial field of type $(1,1)$, and $\mu$ is a $1$-form over $M$.
Moreover, $\mu(Y)=g(PY, N)$.

Let $\nabla$ be the Levi-Civita connection of the induced metric $g$ on $M$. The Gauss
and Weingarten formulae say that
\begin{align*}
\bar{\nabla}_Y Z=\nabla_Y Z+g(AY,Z)N, \quad \bar{\nabla}_Y N=-AY,
\end{align*}
where $A$ is the shape operator of $M$.

The Gauss and Codazzi equations of $M$ are given by
\begin{equation}\label{eqn:2.1}
\begin{aligned}
R(U, Y)Z=&\tfrac{1}{2}\big[g(Y,Z)U-g(U,Z)Y
+g(TY,Z)TU-g(TU,Z)TY\big]\\
&+g(AY,Z)AU-g(AU,Z)AY,
\end{aligned}
\end{equation}
\begin{equation}\label{eqn:2.2}
(\nabla_YA)Z-(\nabla_Z A)Y=\tfrac{1}{2}\big[g(Y,X)TZ
-g(Z, X)TY\big],
\end{equation}
where $U, Y, Z\in TM$, and $R$ denotes the curvature tensor of $M$ with respect
to the metric $g$.
Thus the Ricci curvature tensor is given by
\begin{equation}\label{eqn:2.3}
\operatorname{Ric}(Y)=\tfrac12\big(Y-CTY+g(Y,X)X\big)+3HAY-A^2Y,
\end{equation}
where $H=\frac{1}{3}{\rm Tr}A$ is the mean curvature of the hypersurface $M$.
It follows that the scalar curvature $\rho$ of $M$ is given by
\begin{equation}\label{eqn:scalar}
\rho=2+9H^{2}-\|A\|^{2}.
\end{equation}

Let $\nabla^2 A$ denote the second covariant derivative of $A$, i.e.,
$$
(\nabla^{2}A)(U,Y,Z):=\nabla_U[(\nabla_Y A)Z]-(\nabla_{\nabla_U Y}
A)Z-(\nabla_Y A){\nabla_U Z}.
$$
Then the Ricci identity states that
\begin{equation}\label{eqn:ric}
\begin{split}
g&\big((\nabla^{2}A)(U,Y,Z),W\big)-g\big((\nabla^{2}A)(Y,U,Z),W\big)\\
&=-g\big(R(U,Y)Z,AW\big)-g\big(R(U,Y)W,AZ\big).
\end{split}
\end{equation}

The following lemma, obtained in \cite{G-H-M-Y,Ur}, describes some properties
of the function $C$, the vector field $X$ and the covariant derivatives of $T$ and $\mu$,
which will be used later.
\begin{lemma}[cf. Lemma 1 of \cite{Ur} and Lemma 2.1 of \cite{G-H-M-Y}]\label{lemma:2.1}
Let $M\hookrightarrow\mathbb{S}^2\times\mathbb{S}^2$ be an
orientable hypersurface and $A$ the shape operator associated
to the unit normal field $N$. Then the gradient of $C$, the
covariant derivatives of $X$, $T$ and $\mu$ are given by
\begin{equation}\label{eqn:2.4}
\begin{split}
&\nabla C=-2 AX,\ \ \nabla_{Y}X=C AY -T A Y,\\
&\left(\nabla_Y T\right) Z =g(AY,Z)X+\mu(Z) A Y,\\
&\left(\nabla_Y \mu\right) Z=C g(AY,Z)-g(TZ,AY).
\end{split}
\end{equation}
\end{lemma}
\begin{proof}
The formulae of the gradient of $C$ and the covariant derivative of $X$ can be directly found
in Lemma 1 of \cite{Ur}. For the rest two formulae, by the definitions of $T$, $\mu$
and the Gauss and Weingarten formulae, we get
\begin{align*}
\left(\nabla_{Y} T\right) Z &=(\bar{\nabla}_{Y} (T Z))^{\top}-T (\nabla_{Y} Z) \\
&=[\bar{\nabla}_{Y}(P Z-\mu(Z) N)]^{\top}-T (\nabla_{Y} Z) \\
&=[P \bar{\nabla}_{Y} Z-\mu(Z) \bar{\nabla}_{Y} N]^{\top}-T (\nabla_{Y} Z) \\
&=g(AY,Z)(P N)^{\top}+\mu(Z) A Y\\
&=g(AY,Z)X+\mu(Z) A Y,
\end{align*}
\begin{align*}
\left(\nabla_{Y} \mu\right) Z&= Y \mu(Z)-\mu\left(\nabla_{Y} Z\right) \\
&=g(\bar{\nabla}_{Y} PZ,N)+g(PZ,\bar{\nabla}_{Y} N)-g(P\bar{\nabla}_{Y} Z-g(AY,Z) PN,N)\\
&=-g(TZ,AY)+Cg(AY,Z),
\end{align*}
where $\cdot^\top$ denotes the tangential part.
\end{proof}

In \cite{Ur}, Urbano proved that hypersurface with $C^2\equiv 1$ is locally
the following hypersurfaces.
\begin{lemma}[\cite{Ur}]\label{lem:2.2}
Let $M$ be a hypersurface of $\mathbb{S}^2\times\mathbb{S}^2$ with $C^2\equiv 1$.
Then, up to isometries of $\mathbb{S}^2\times\mathbb{S}^2$, $M$ is locally the
product of a curve $\Gamma$ in $\mathbb{S}^2$ and an open
subset of $\mathbb{S}^2$.
\end{lemma}

\subsection{A canonical frame related to hypersurfaces of
$\mathbb{S}^2\times\mathbb{S}^2$}\label{sect:2.3}~

In this subsection, we assume that $|C|< 1$ holds on $M$,
and choose an appropriate local orthonormal frame fields $\{E_1, E_2, E_3\}$
of $M$ as follows:
\begin{equation}\label{eqn:2.5}
E_1=\frac{J_1N+J_2 N}{\sqrt{2(1+C)}},\
\ E_2=\frac{J_1N-J_2 N}{\sqrt{2(1-C)}},\ \ E_3=\frac{X}{\sqrt{1-C^{2}}}.
\end{equation}

This frame has the following properties:
\begin{equation}\label{eqn:2.6}
\left\{
\begin{aligned}
&P E_1=E_1,\ \ P E_2=-E_2,\ \ P E_3=\sqrt{1-C^{2}}N-C E_3,\ \ PN=CN+\sqrt{1-C^{2}}E_3.\\
&TE_1=E_1,\ \ TE_2=-E_2,\ \ TE_3=-CE_3,\ \ \mu(E_1)=\mu(E_2)=0,\ \ \mu(E_3)=\sqrt{1-C^{2}}.
\end{aligned}\right.
\end{equation}

We assume that
\begin{equation}\label{eqn:2.7}
\begin{aligned}	
&A E_1=b_1 E_1+b_2 E_2+b_3 E_3, \\
&A E_2=b_2 E_1+b_4 E_2+b_5 E_3,\\
&A E_3=b_3 E_1+b_5E_2+b_6 E_3,
\end{aligned}
\end{equation}
where $b_1$, $b_2$, $b_3$, $b_4$, $b_5$ and $b_6$ are smooth
functions on $M$.

By \eqref{eqn:2.3}, under the frame $\{E_1,E_2,E_3\}$, the Ricci tensor is expressed by
\begin{equation}\label{eqn:2.8}
\begin{aligned}
&{\rm Ric}(E_1)=(\tfrac{1-C}{2}
+b_1b_4+b_1b_6-b_2^{2}-b_3^{2}) E_1+
(b_2b_6-b_3 b_5) E_2+(b_3b_4-b_2b_5) E_3, \\
&{\rm Ric}(E_2)=(b_2b_6-b_3 b_5) E_1
+(\tfrac{1+C}{2}+b_1b_4+b_4b_6-b_2^{2}-b_5^{2}) E_2+(b_1b_5-b_2b_3) E_3,\\
&{\rm Ric}(E_3)=(b_3b_4-b_2b_5) E_1+(b_1b_5-b_2b_3) E_2+(1+b_1b_6+b_4b_6-b_3^{2}-b_5^{2}) E_3.
\end{aligned}
\end{equation}

When $C$ is constant which is not equal to $\pm1$ on $M$, we have the following lemma:

\begin{lemma}[cf. P.1395 of \cite{Ur}]\label{lemma:2.3}
Let $M$ be a hypersureface of $\mathbb{S}^2\times \mathbb{S}^2$ with constant $C\neq\pm1$.
Then, under the local orthonormal frame fields $\{E_1, E_2, E_3\}$,
the following relations hold:
\begin{equation}\label{eqn:2.25}
\begin{aligned}
&b_3=b_5=b_6=0, \ \ A E_1=b_1 E_1+b_2 E_2,\ A E_2=b_2 E_1+b_4 E_2,\ A E_3=0,\\		
&\nabla_{E_1}E_1=b_1\sqrt{\tfrac{1-C}{1+C}}E_3,\
\nabla_{E_1}E_2=-b_2\sqrt{\tfrac{1+C}{1-C}}E_3,\ \nabla_{E_1}E_3=-b_1\sqrt{\tfrac{1-C}{1+C}}E_1
+b_2\sqrt{\tfrac{1+C}{1-C}}E_2,\\
&\nabla_{E_2} E_1=b_2\sqrt{\tfrac{1-C}{1+C}}E_3,\
\nabla_{E_2} E_2=-b_4\sqrt{\tfrac{1+C}{1-C}}E_3,\ \nabla_{E_2}E_3=-b_2\sqrt{\tfrac{1-C}{1+C}}E_1
+b_4\sqrt{\tfrac{1+C}{1-C}}E_2,\\
&\nabla_{E_3}E_1=\nabla_{E_3}E_2=\nabla_{E_3}E_3=0.
\end{aligned}
\end{equation}
Moreover, the Gauss and Codazzi equations are equivalent to the following
equations:
\begin{equation}\label{eqn:2.26}
E_3b_1=\tfrac{\sqrt{1-C^2}}{2}+b_1^{2}\sqrt{\tfrac{1-C}{1+C}}-b_2^2\sqrt{\tfrac{1+C}{1-C}},
\end{equation}
\begin{equation}\label{eqn:2.27}
E_3b_2=b_2\big(b_1\sqrt{\tfrac{1-C}{1+C}}-b_4\sqrt{\tfrac{1+C}{1-C}}\big),
\end{equation}
\begin{equation}\label{eqn:2.28}
E_3b_4=-\tfrac{\sqrt{1-C^2}}{2}+b_2^{2}\sqrt{\tfrac{1-C}{1+C}}-b_4^2\sqrt{\tfrac{1+C}{1-C}},
\end{equation}
\begin{equation}\label{eqn:2.29}
E_1b_2-E_2b_1=0,	
\end{equation}
\begin{equation}\label{eqn:2.30}
E_1b_4-E_2b_2=0.
\end{equation}
\end{lemma}

\begin{remark}
	Due to the differences between the subscripts of the frame, the formulae in Lemma \ref{lemma:2.3} are slightly different from that on Page 1395 of \cite{Ur}.
	\end{remark}

\section{Minimal hypersurfaces with $C=0$}\label{sect:3}~
In this section, we  give the classification of the minimal hypersurfaces with $C=0$,
which plays an very important role in the proof of Theorem \ref{thm:1.1} and Theorem \ref{thm:1.2}.

\begin{lemma}\label{lemma:2.1aaa1}
Let $M$ be a minimal hypersureface of $\mathbb{S}^2\times \mathbb{S}^2$ with $C=0$.
Then, under the frame $\{E_1,E_2,E_3\}$ defined by \eqref{eqn:2.5}, we have three cases
of $b_1,b_2$ and $b_4$ on $M$ as follows.

\noindent{\bf Case 1}:
$$
b_1=b_4=0,\ \ b_2=\tfrac{1}{\sqrt{2}},
$$

\noindent{\bf Case 2}:
$$
b_1=-b_4=\tfrac{1}{\sqrt{2}}\tan(\tfrac{t-k}{\sqrt{2}}),\ \ b_2=0,
$$

\noindent{\bf Case 3}:
\begin{equation}\label{eq:j}
\begin{aligned}
b_1&=-b_4=\tfrac{1}{\sqrt{2}}\frac{\sin(\frac{t-l}{\sqrt{2}})\cos(\frac{t-l}{\sqrt{2}})}
{\cos^2(\frac{t-l}{\sqrt{2}})+\sinh^2(\frac{h}{\sqrt{2}})},\\
b_2&=\tfrac{1}{\sqrt{2}}\frac{\sinh(\frac{h}{\sqrt{2}})\cosh(\frac{h}{\sqrt{2}})}
{\cos^2(\frac{t-l}{\sqrt{2}})+\sinh^2(\frac{h}{\sqrt{2}})},
\end{aligned}	
\end{equation}
where the  function $t$ satisfies $E_3t=1$ and the functions $h, k, l:M\rightarrow R$
satisfy $h\not\equiv 0$ and $E_3h=E_3k=E_3l=0$.

\end{lemma}
\begin{proof}
By using $H=\frac{1}{3}(b_1+b_4)=0$ and $C=0$, the equations \eqref{eqn:2.26} and \eqref{eqn:2.27} become
\begin{equation*}
\begin{split}
E_3b_1&=\tfrac{1}{2}+b_1^2-b_2^2,\\
E_3b_2&=2b_1b_2.
\end{split}	
\end{equation*}

When $b_2$ is a constant on $M$, we have $b_1b_2=0$. Furthermore,
if $b_1=0$, then up to a sign of the unit normal $N$, we are in the situation of  {\bf Case 1}.
If $b_2=0$, then $b_1$ satisfies $E_3b_1=\tfrac{1}{2}+b_1^2$. By integrating
this equation along the integral curves of $E_3$, we  obtain {\bf Case 2}.

When $b_2$ is not a constant on $M$, let $F=b_1+\mathbf{i}b_2$, we get
$$
E_3F=E_3(b_1+\mathbf{i}b_2)=\frac{1}{2}+(b_1+\mathbf{i}b_2)^2=\frac{1}{2}+F^2.
$$
This equation can be directly integrated along the integral curves of $E_3$
to obtain $F=\frac{1}{\sqrt{2}}\tan(\frac{t+z_0}{\sqrt{2}})$,
where $z_0=-l+\mathbf{i}h$, $h$ and $l$ are functions on $M$ which satisfy $h\not\equiv 0$ and $E_3h=E_3l=0$.
By $F=b_1+\mathbf{i}b_2$, it follows that $b_1$ and $b_2$ are the real part and the imaginary part of $F$, respectively,
which corresponds to  {\bf Case 3}.
\end{proof}

In the following, we deal with {\bf Case 3} of Lemma \ref{lemma:2.1aaa1}.
Inspired by the techniques in \cite{K-V}, we use the integrability conditions to find appropriate local coordinates.
We assume that

\begin{equation}\label{UVT}
\left(
  \begin{array}{c}
    U \\
    V \\
    T \\
  \end{array}
\right)
=
\left(
  \begin{array}{ccc}
    a_1 & a_2 & a_1d_1+a_2d_2 \\
    -a_2 & a_1 & -a_2d_1+a_1d_2 \\
    0 & 0 & 1 \\
  \end{array}
\right)
\left(
  \begin{array}{c}
    E_1 \\
    E_2 \\
    E_3 \\
  \end{array}
\right),
\end{equation}
where $a_1,a_2,d_1,d_2$ are smooth functions on $M$.

\begin{lemma}\label{lemma:2.1a1}
 $\{U,V,T\}$ is a local coordinate frame  if and only if
the following system of equations is satisfied:
\begin{equation}\label{eq:2.ab1}
E_3a_1=-a_1b_1-a_2b_2,\ \ E_3a_2=a_1b_2-a_2b_1,
\end{equation}
\begin{equation}\label{eq:2.ab2}
E_1a_2+E_2a_1=0,\ \ E_1a_1-E_2a_2=0,
\end{equation}
\begin{equation}\label{eq:2.ab3}
E_3d_1=b_1d_1-b_2d_2,\ \ E_3d_2=b_1d_2+b_2d_1,
\end{equation}
\begin{equation}\label{eq:2.ab4}
E_1d_2-E_2d_1=2b_2.
\end{equation}

\end{lemma}
\begin{proof}
We will show that our system of equations is equivalent to that all the commutators $[T,U]$, $[T,V]$ and $[U,V]$ vanish. Using $[E_i,E_j]=\nabla_{E_i} {E_j}-\nabla_{E_j} {E_i}$, $C=0$
and \eqref{eqn:2.25}, we get
\begin{equation}\label{liebracket}
[E_1,E_2]=-2b_2E_3,\ [E_1,E_3]=-b_1E_1+b_2E_2,\ [E_2,E_3]=-b_2E_1+b_4E_2=-b_2E_1-b_1E_2.
\end{equation}

Therefore
$$
\begin{aligned}
{\text [T,U]}=&(E_3a_1+a_1b_1+a_2b_2)E_1+(E_3a_2-a_1b_2+a_2b_1)E_2\\
&+(d_1 E_3a_1+a_1E_3d_1+d_2E_3a_2+a_2E_3d_2)E_3.
\end{aligned}
$$
Then $[T,U]=0$ is equivalent to \eqref{eq:2.ab1} and
\begin{equation}\label{eq:2.ab5}
a_1E_3d_1+a_2E_3d_2=(a_1b_1+a_2b_2)d_1+(a_2b_1-a_1b_2)d_2.
\end{equation}

Similarly, we get
$$
\begin{aligned}
{\text [T,V]}=&(-E_3a_2-a_2b_1+a_1b_2)E_1+(E_3a_1+a_2b_2+a_1b_1)E_2\\
&+(-d_1 E_3a_2-a_2E_3d_1+d_2E_3a_1+a_1E_3d_2)E_3.
\end{aligned}
$$
Using \eqref{eq:2.ab1}, we have the following equation
\begin{equation}\label{eq:2.ab6}
-a_2E_3d_1+a_1E_3d_2=(a_1b_2-a_2b_1)d_1+(a_1b_1+a_2b_2)d_2.
\end{equation}

A short calculation implies that \eqref{eq:2.ab5} and \eqref{eq:2.ab6}
are equivalent to \eqref{eq:2.ab3}. Finally, a straightforward calculation gives
$$
\begin{aligned}
{\text [U,V]}=&\Big(-a_1(E_1a_2+E_2a_1)+a_2(E_1a_1-E_2a_2)\Big)E_1\\
&+\Big(a_2(E_1a_2+E_2a_1)+a_1(E_1a_1-E_2a_2)\Big)E_2\\
&+\Big((a_1^2+a_2^2)(E_1d_2-E_2d_1-2b_2)-a_1d_1(E_1a_2+E_2a_1)+a_2d_1(E_1a_1-E_2a_2)\\
&+a_2d_2(E_1a_2+E_2a_1)+a_1d_2(E_1a_1-E_2a_2)\Big)E_3.
\end{aligned}
$$
Since $a_1^2+a_2^2\neq0$, we obtain that $[U,V]$ vanishes
 if and only if both \eqref{eq:2.ab2} and \eqref{eq:2.ab4} hold.
\end{proof}

Any solution to the system of equations \eqref{eq:2.ab1}--\eqref{eq:2.ab4}
 provides us  adapted coordinates. We present a special solution as follows.

\begin{lemma}\label{lemma:2.1a2}
Assume that $b_1,b_2$ and $b_4$ are given by \eqref{eq:j} on $M$,
let $f$ be any function satisfying $E_3f=-1$. Then the functions $a_1,a_2,d_1,d_2$ defined by
\begin{equation}\label{eq:2.ab7}
\begin{aligned}
(a_1+\mathbf{i}a_2)^2=\frac{1}{2(b_1-\mathbf{i}b_2)^2+1}, \ \ d_1=E_1f,\ d_2=E_2f,
\end{aligned}
\end{equation}
are a solution to the system of differential equations \eqref{eq:2.ab1}--\eqref{eq:2.ab4}.
\end{lemma}
\begin{proof}
Observe that the system of equations splits into two independent
subsystems, \eqref{eq:2.ab1}-\eqref{eq:2.ab2} and \eqref{eq:2.ab3}-\eqref{eq:2.ab4}.

Assume that $a_1,a_2$ are functions of $b_1,b_2$. Using the equations
\eqref{eqn:2.26}--\eqref{eqn:2.30} and $C=0$, we express the first subsystem \eqref{eq:2.ab1}-\eqref{eq:2.ab2}  with respect to
 $b_1,b_2$ and  obtain that
\begin{equation}\label{eq:2.ab8}
(b_1^2-b_2^2+\frac{1}{2})\frac{\partial a_1}{\partial b_1}+2b_1b_2\frac{\partial a_1}{\partial b_2}
=-a_1b_1-a_2b_2,
\end{equation}
\begin{equation}\label{eq:2.ab9}
(b_1^2-b_2^2+\frac{1}{2})\frac{\partial a_2}{\partial b_1}+2b_1b_2\frac{\partial a_2}{\partial b_2}
=a_1b_2-a_2b_1,
\end{equation}
\begin{equation}\label{eq:2.ab10}
\frac{\partial a_2}{\partial b_1}E_1b_1+\frac{\partial a_2}{\partial b_2}E_1b_2
+\frac{\partial a_1}{\partial b_1}E_2b_1+\frac{\partial a_1}{\partial b_2}E_2b_2=0,
\end{equation}
\begin{equation}\label{eq:2.ab11}
\frac{\partial a_1}{\partial b_1}E_1b_1+\frac{\partial a_1}{\partial b_2}E_1b_2
-\frac{\partial a_2}{\partial b_1}E_2b_1-\frac{\partial a_2}{\partial b_2}E_2b_2=0.
\end{equation}
Equations $\eqref{eq:2.ab10}-\eqref{eq:2.ab11}$ are satisfied if $(a_1,a_2)$ satisfy the Cauchy-Riemann equations
with respect to $b_1,b_2$, i.e.,
$$
\frac{\partial a_2}{\partial b_1}=\frac{\partial a_1}{\partial b_2},\ \
\frac{\partial a_2}{\partial b_2}=-\frac{\partial a_1}{\partial b_1}.
$$
We assume that $a_1(z)+\mathbf{i}a_2(z)$ is a holomorphic function with respect
to $z=b_1-\mathbf{i}b_2$, then
$$
(z^2+\tfrac{1}{2})\frac{\partial}{\partial z}(a_1+\mathbf{i}a_2)=
(b_1^2-b_2^2+\tfrac{1}{2})\frac{\partial a_1}{\partial b_1}+2b_1b_2\frac{\partial a_1}{\partial b_2}
+\mathbf{i}\Big((b_1^2-b_2^2+\tfrac{1}{2})\frac{\partial a_2}{\partial b_1}+2b_1b_2\frac{\partial a_2}{\partial b_2}\Big),
$$
and
$$
-z(a_1+\mathbf{i}a_2)=(-a_1b_1-a_2b_2)+\mathbf{i}(-b_1a_2+b_2a_1).
$$
Then \eqref{eq:2.ab8} and \eqref{eq:2.ab9} are equivalent to
$$
(z^2+\tfrac{1}{2})\frac{\partial}{\partial z}(a_1+\mathbf{i}a_2)=-z(a_1+\mathbf{i}a_2).
$$
Thus, we obtain a special solution to the equation system  $\eqref{eq:2.ab8}-\eqref{eq:2.ab11}$: $(a_1+\mathbf{i}a_2)^2=\frac{1}{2(b_1-\mathbf{i}b_2)^2+1}$.

Let $f$ be any function satisfying that $E_3f=-1$. We denote $d_1=E_1f, d_2=E_2f$, by using \eqref{liebracket}, we obtain the following integrability conditions for $f$.
$$
\begin{aligned}
	&0=[E_1,E_3]f+b_1E_1(f)-b_2E_2(f)=-E_3d_1+b_1d_1-b_2d_2,\\
	&0=[E_2,E_3]f+b_2E_1(f)+b_1E_2(f)=-E_3d_2+b_2d_1+b_1d_2,\\
	&0=[E_1,E_2]f+2b_2E_3(f)=E_1d_2-E_2d_1-2b_2,
\end{aligned}
$$
which immediately implies that  $d_1=E_1f, d_2=E_2f$ is a solution to the
second subsystem \eqref{eq:2.ab3}-\eqref{eq:2.ab4}.
\end{proof}

We denote the coordinates obtained from the solution given in Lemma \ref{lemma:2.1a2} by $(u,v,t)$, i.e.,
$\partial_u=U,\ \partial_v=V,\ \partial_t=T=E_3$.
We will further make a special choice for $f$ to simplify \eqref{eq:j}.

\begin{lemma}\label{lemma:2.1a3}
There exists a function $f$ with $E_3f=-1$ and a constant $c$ such that
replacing $t$ by $t-c$ one has $l=0$. Moreover, $f$ is given by
$$
f=-\tfrac{1}{\sqrt{2}}\arcsin\Big(\tfrac{\sqrt{2}b_1}
{\sqrt{(b_1^2-b_2^2+\frac{1}{2})^2+4b_1^2b_2^2}}\Big).
$$
The variable $t$ is given by $t=-f+$const.

\end{lemma}
\begin{proof}
A short calculation implies that our choice of $f$ satisfies $E_3f=-1$.
Using the defining equation \eqref{UVT} for $T,U,V$,  we get
\begin{equation}\label{eq:2.ab12}
E_1=\frac{a_1}{a_1^2+a_2^2}\partial_u-\frac{a_2}{a_1^2+a_2^2}\partial_v-d_1\partial_t,
\end{equation}
\begin{equation}\label{eq:2.ab13}
E_2=\frac{a_2}{a_1^2+a_2^2}\partial_u+\frac{a_1}{a_1^2+a_2^2}\partial_v-d_2\partial_t,
\end{equation}
\begin{equation}\label{eq:2.ab14}
E_3=\partial_t.
\end{equation}
From $E_1f=d_1, E_2f=d_2, E_3f=-1$ we have
$$
\frac{\partial}{\partial u}f=\frac{\partial}{\partial v}f=0.
$$
We insert \eqref{eq:j} into the definition of $f$ and obtain $\sin(-\sqrt{2}f)=\sin(\sqrt{2}(t-l))$, hence $\frac{\partial}{\partial u}l=\frac{\partial}{\partial v}l=0$.
Recall that $E_3l=0$ (see Lemma \ref{lemma:2.1aaa1}), we obtain that $l$ must be a constant.
\end{proof}

\begin{theorem}\label{thm:3.1}
Let $M$ be a minimal hypersurface of $\mathbb{S}^2\times \mathbb{S}^2$ with product angle function $C=0$.
Assume that with respect to frame field $\{E_1,E_2,E_3\}$ (see \eqref{eqn:2.5}), the function $b_2$ defined in \eqref{eqn:2.7} is not constant on $M$.
Then $M$ admits local coordinates $(u,v,t)$ and there is a non-zero function $h:(u,v)\mapsto \mathbb{R}$
such that  the
functions $b_1,b_2,b_4$ defined in \eqref{eqn:2.7}  are given by
\begin{equation}\label{eq:2.ab15}
b_1=-b_4=\tfrac{1}{\sqrt{2}}\frac{\sin(\frac{t}{\sqrt{2}})\cos(\frac{t}{\sqrt{2}})}
{\cos^2(\frac{t}{\sqrt{2}})+\sinh^2(\frac{h}{\sqrt{2}})},\ \
b_2=\tfrac{1}{\sqrt{2}}\frac{\sinh(\frac{h}{\sqrt{2}})\cosh(\frac{h}{\sqrt{2}})}
{\cos^2(\frac{t}{\sqrt{2}})+\sinh^2(\frac{h}{\sqrt{2}})},
\end{equation}
where $h$ satisfies the following ``sinh-Gordon equation'':
\begin{equation}\label{eq:2.ab16}
(\frac{\partial^2}{\partial u^2}+\frac{\partial^2}{\partial v^2})h
=-\tfrac{1}{\sqrt{2}}\sinh(\sqrt{2}h).
\end{equation}

Conversely, let $\mathcal{D}$ be an open subset in $\mathbb{R}^2$ and  $h:\mathcal{D}\to\mathbb{R}, (u,v)\mapsto h(u,v)\in\mathbb{R}$
be a non-zero function which satisfies the differential equation \eqref{eq:2.ab16}.
Let $\Omega_0=\{(u,v,t)\in \mathcal{D}\times \mathbb{R}|h(u,v)=0,\sqrt{2}t=(2p+1)\pi, p\in\mathbb{Z}\}$, $\Omega=\{(u,v,t)\in \mathcal{D}\times \mathbb{R}\}-\Omega_0\subseteq \mathbb{R}^3$, we define a metric $g$ on $\Omega$ by
\begin{equation}\label{eq:2.ab17}
\begin{aligned}
g&=\Big(\tfrac{1}{2}(\cos(\sqrt{2}t)+\cosh(\sqrt{2}h))+h_v^2\Big)du^2
+\Big(\tfrac{1}{2}(\cos(\sqrt{2}t)+\cosh(\sqrt{2}h))+h_u^2\Big)dv^2+dt^2\\
&-2h_uh_vdudv+2h_vdudt-2h_udvdt,
\end{aligned}
\end{equation}
and a $(1,1)$-tensor field $A$ on $T\Omega$ by
\begin{equation}\label{eqn:A}
\begin{aligned}
A\partial_u&=\tfrac{\cosh(\sqrt{2}h)\sin(\sqrt{2}t)}{\sqrt{2}(\cos(\sqrt{2}t)+\cosh(\sqrt{2}h))}
\partial_u
+\tfrac{\cos(\sqrt{2}t)\sinh(\sqrt{2}h)}{\sqrt{2}(\cos(\sqrt{2}t)+\cosh(\sqrt{2}h))}\partial_v\\
&\ \ -\tfrac{\cosh(\sqrt{2}h)\sin(\sqrt{2}t)h_v-\cos(\sqrt{2}t)\sinh(\sqrt{2}h)h_u}
{\sqrt{2}(\cos(\sqrt{2}t)+\cosh(\sqrt{2}h))}\partial_t,\\
A\partial_v&=\tfrac{\cos(\sqrt{2}t)\sinh(\sqrt{2}h)}{\sqrt{2}(\cos(\sqrt{2}t)+\cosh(\sqrt{2}h))}
\partial_u
-\tfrac{\cosh(\sqrt{2}h)\sin(\sqrt{2}t)}{\sqrt{2}(\cos(\sqrt{2}t)+\cosh(\sqrt{2}h))}\partial_v\\
&\ \ -\tfrac{\sinh(\sqrt{2}h)\cos(\sqrt{2}t)h_v+\sin(\sqrt{2}t)\cosh(\sqrt{2}h)h_u}
{\sqrt{2}(\cos(\sqrt{2}t)+\cosh(\sqrt{2}h))}\partial_t,\\
A\partial_t&=0.
\end{aligned}
\end{equation}
Let $\nu$ be a vecor bundle over  $\Omega$ of rank $1$ with metric $\tilde{g}$, $\nabla ^{\perp} $ a connection on $\nu$ compatible with the metric $\tilde{g}$, and $N$ a unit vector field in $\nu$.   We define a $(1,1)$-tensor field $\tilde{P}$ on $T\Omega\otimes \nu$ by
\begin{equation}\label{defP}
\begin{aligned}
	\tilde{P}\partial_u&=\tfrac{1+\cos(\sqrt{2}t)\cosh(\sqrt{2}h)}{\cos(\sqrt{2}t)+\cosh(\sqrt{2}h)}
	\partial_u
	-\tfrac{\sin(\sqrt{2}t)\sinh(\sqrt{2}h)}{\cos(\sqrt{2}t)+\cosh(\sqrt{2}h)}\partial_v\\
	&\ \ -
\tfrac{h_v+h_v\cos(\sqrt{2}t)\cosh(\sqrt{2}h)+h_u\sin(\sqrt{2}t)\sinh(\sqrt{2}h)}{{\cos(\sqrt{2}t)+\cosh(\sqrt{2}h)}}\partial_t+h_vN\\
	\tilde{P}\partial_v&=-\tfrac{\sin(\sqrt{2}t)\sinh(\sqrt{2}h)}{\cos(\sqrt{2}t)+\cosh(\sqrt{2}h)}
	\partial_u-
	\tfrac{1+\cos(\sqrt{2}t)\cosh(\sqrt{2}h)}{\cos(\sqrt{2}t)+\cosh(\sqrt{2}h)}\partial_v\\
	&\ \  -\tfrac{h_u+h_u\cos(\sqrt{2}t)\cosh(\sqrt{2}h)-h_v\sin(\sqrt{2}t)\sinh(\sqrt{2}h)}{{\cos(\sqrt{2}t)+\cosh(\sqrt{2}h)}}\partial_t-h_uN,\\
	\tilde{P}\partial_t&=N,\\
	\tilde{P}N&=\partial_t.\\	
\end{aligned}
\end{equation}
Then the integrability conditions are satisfied on $\Omega$ and hence on any simply
connected subset $\Omega_1$ of $\Omega$, up to an isometry of  $\mathbb{S}^2\times \mathbb{S}^2$, there is a unique isometric immersion $\Psi$ from $(\Omega_1,~g)$ into $\mathbb{S}^2\times \mathbb{S}^2$ such that $A$ is the shape operator of  $\Omega_1$ into $\mathbb{S}^2\times \mathbb{S}^2$, $\nu$ is isometric to the normal bundle of $\Psi(\Omega_1)$ in
$\mathbb{S}^2\times \mathbb{S}^2$ by an isomorphism $\tilde{\Psi}:\nu\to T^{\perp}\Psi(\Omega_1)$,
and for any $Y\in T\Omega_1$, we have
\begin{equation}
	P(\Psi_{*}Y)=\Psi_{*}((\tilde{P}Y)^{T})+\tilde{\Psi}((\tilde{P}Y)^{\perp}),~
	P(\tilde{\Psi}N)=\Psi_*(\partial_t),
\end{equation}
where $P$ is the product structure of $\mathbb{S}^2\times \mathbb{S}^2$, $(\tilde{P}Y)^{T}$ and $(\tilde{P}Y)^{\perp}$ denote the projections of  $\tilde{P}Y$ onto $T\Omega_1$ and $\nu$, respectively.
Moreover,  $\Psi$ is minimal with the product angle function $C=0$.
\end{theorem}
\begin{proof}
Let $M$ be a minimal hypersurface of $\mathbb{S}^2\times \mathbb{S}^2$ with $C=0$.
By Lemmas \ref{lemma:2.1aaa1}-\ref{lemma:2.1a3},  $M$ admits local coordinates $(u,v,t)$ and  
 there is a function $h:(u,v)\mapsto \mathbb{R}$
such that $b_1,b_2,b_4$ are given by \eqref{eq:2.ab15}.
From Lemma \ref{lemma:2.1a2}, we  obtain that
\begin{equation}\label{eq:2.ab18}
a_1=\cos(\tfrac{t}{\sqrt{2}})\cosh(\tfrac{h}{\sqrt{2}}),\
a_2=\sin(\tfrac{t}{\sqrt{2}})\sinh(\tfrac{h}{\sqrt{2}}).
\end{equation}
Using \eqref{eq:2.ab12},\eqref{eq:2.ab13} and \eqref{eq:2.ab15},
we get that  \eqref{eqn:2.29} and \eqref{eqn:2.30} are equivalent to
$$
d_1=\tfrac{2\big(\cos(\frac{t}{\sqrt{2}})\cosh(\frac{h}{\sqrt{2}})h_v
+\sin(\frac{t}{\sqrt{2}})\sinh(\frac{h}{\sqrt{2}})h_u\big)}{\cos(\sqrt{2}t)+\cosh(\sqrt{2}h)},\ \
d_2=\tfrac{2\big(\sin(\frac{t}{\sqrt{2}})\sinh(\frac{h}{\sqrt{2}})h_v
-\cos(\frac{t}{\sqrt{2}})\cosh(\frac{h}{\sqrt{2}})h_u\big)}{\cos(\sqrt{2}t)+\cosh(\sqrt{2}h)}.
$$
Then we can check the compatibility conditions of  $b_1,b_2,b_4$ with respect to the Lie brackets $[E_i,E_j]$, $1\leq i<j\leq3$,
and obtain that
$$
(\frac{\partial^2}{\partial u^2}+\frac{\partial^2}{\partial v^2})h
=-\tfrac{1}{\sqrt{2}}\sinh(\sqrt{2}h).
$$
In fact, the  compatibility conditions of  $b_1,b_2,b_4$  with respect to the Lie brackets $[E_1,E_3]$ and $[E_2,E_3]$ are automatically satisfied and the  compatibility conditions of  $b_1,b_2,b_4$  with respect to the Lie bracket $[E_1,E_2]$  imply the above differential equation.

Conversely, assume that the function $h$ satisfies \eqref{eq:2.ab16}, the metric $g$ is defined according
to \eqref{eq:2.ab17}, the $(1,1)$-tensor $A$ is defined by \eqref{eqn:A}, and we have a $(1,1)$-tensor $\tilde{P}$ defined by \eqref{defP},
then it can be  checked  directly that all the integrability conditions are satisfied.
Therefore, the conclusions can be obtained by applying  the existence and uniqueness theorems in  \cite{K,L-T-V} directly.
\end{proof}

\begin{remark}\label{rem:3.1aa}
(i) According to Corollary 1 (3)
of \cite{Ur}, $M_0$ is the only hypersurface of $\mathbb{S}^2\times\mathbb{S}^2$
which is minimal and has the scalar curvature and the function $C$ constants. Under the frame $\{E_1,E_2,E_3\}$
	defined in \eqref{eqn:2.5}, $b_2=\frac{1}{\sqrt{2}}$
	on $M_0$.
Hence, if $M$ is a minimal hypersurface described in Theorem \ref{thm:3.1},
then $M$ can not have constant scalar curvature.
(ii) {\bf Case 2} of Lemma \ref{lemma:2.1aaa1}
corresponds to a trivial solution of ``sinh-Gordon equation'': $h\equiv0$,
and the coresponding example is the hypersurface
 $M_{a,b}$, see Theorem  \ref{thm:3.2aa}
 for more details.
\end{remark}

Before dealing with {\bf Case 1} and {\bf Case 2} of Lemma \ref{lemma:2.1aaa1},
we present the following  classification
result for  hypersurfaces of $\mathbb{S}^2\times\mathbb{S}^2$
with constant $C\neq\pm 1$ and constant $b_2$.
 Meanwhile, for any given constant  $C\neq\pm1$, we construct a large family of new examples of hypersurfaces in $\mathbb{S}^2\times\mathbb{S}^2$
 with constant product angle function $C$.

\begin{proposition}\label{prop:3.61}
Let $M$ be a hypersurface with constant $C\neq\pm1$ in $\mathbb{S}^2\times\mathbb{S}^2$.
Then, under the frame $\{E_1,E_2,E_3\}$ defined in \eqref{eqn:2.5}, $b_2$ is constant on $M$ if and only if either
\begin{enumerate}
\item[(1)]
$b_2\neq 0$ and $M$ is an open part of $M_t$ for some $t\in(-1,1)$, or

\item[(2)]
$b_2=0$ and $M$ is locally given by the  immersion
$\Phi: \Omega\subset\mathbb{R}^{3}\longrightarrow \mathbb{S}^2\times\mathbb{S}^2$:
\ $(t,r,s)\longrightarrow(p,q),$
\begin{equation}\label{eqn:6.1}
\begin{aligned}
&p(t,r)=\cos(\sqrt{\tfrac{1-C}{2}}t)\gamma(r)+\sin(\sqrt{\tfrac{1-C}{2}}t)N(r),\\
&q(t,s)=\cos(\sqrt{\tfrac{1+C}{2}}t)\tilde{\gamma}(s)+\sin(\sqrt{\tfrac{1+C}{2}}t)\tilde{N}(s),
\end{aligned}	
\end{equation}
\end{enumerate}
where $\gamma(r)$ and $\tilde{\gamma}(s)$ are  two arbitrary smooth regular curves in $\mathbb{S}^2$ with
 arc length parameters $r$ and $s$,
and $N(r)$ and $\tilde{N}(s)$ are the unit normal vector fields of $\gamma(r)$ and $\tilde{\gamma}(s)$
in $\mathbb{S}^2$, respectively.
\end{proposition}

\begin{proof}
We assume that $b_2$ is constant on $M$. If $b_2$ is a non-zero constant, then by \eqref{eqn:2.27} we have
$$
b_1\sqrt{\tfrac{1-C}{1+C}}-b_4\sqrt{\tfrac{1+C}{1-C}}=0,
$$
which implies that $b_4=\frac{1-C}{1+C}b_1$.
Substituting $b_4=\frac{1-C}{1+C}b_1$ into \eqref{eqn:2.26} and \eqref{eqn:2.28}, we have
$$
\begin{aligned}
E_3b_1&=-\tfrac{1+C}{2}\sqrt{\tfrac{1+C}{1-C}}+b_2^2\sqrt{\tfrac{1+C}{1-C}}
-b_1^2\sqrt{\tfrac{1-C}{1+C}},\\
E_3b_1&=\tfrac{\sqrt{1-C^2}}{2}+b_1^2\sqrt{\tfrac{1-C}{1+C}}-b_2^2\sqrt{\tfrac{1+C}{1-C}}.
\end{aligned}
$$
It follows from the  above two equations that $b_1$ and $b_4$ are also constant on $M$, hence $M$ has constant principal curvatures. Therefore, $M$ has constant $C$, constant mean curvature and constant scalar curvature.
According to Corollary 1 (3) of \cite{Ur}, $M$ is an open part of $M_t$ for some $t\in(-1,1)$.  Conversely, one can verify directly  that the hypersurfaces $\{M_t,~t\in(-1,1)\}$  have constant $C$ and satisfy that  $b_2$ is a non-zero  constant (cf. Section 3.2 of \cite{Ur}).

In the following, we consider the case with $b_2=0$ on $M$.
By \eqref{eqn:2.26} and \eqref{eqn:2.28}, we have
$$
\begin{aligned}
&E_3b_1=\tfrac{\sqrt{1-C^2}}{2}+b_1^{2}\sqrt{\tfrac{1-C}{1+C}},\\
&E_3b_4=-\tfrac{\sqrt{1-C^2}}{2}-b_4^2\sqrt{\tfrac{1+C}{1-C}}.
\end{aligned}
$$
The above two equations can be  integrated along the integral curves of $E_3$ directly
and we obtain that
\begin{equation}\label{eqn:b1b4}
\begin{aligned}
&b_1=\sqrt{\tfrac{1+C}{2}}\tan\big(\sqrt{\tfrac{1-C}{2}}(t+h_1)\big),\\
&b_4=-\sqrt{\tfrac{1-C}{2}}\tan\big(\sqrt{\tfrac{1+C}{2}}(t+h_2)\big),
\end{aligned}
\end{equation}
where $h_1$ and $h_2$ are two local functions on $M$ which satisfy that $E_3h_1=E_3h_2=0$.

Next, we choose two appropriate non-zero functions $\rho_1$ and $\rho_2$, such that
the new frame
$$
\{X_1=\rho_1E_1,\ X_2=\rho_2E_2,\ X_3=E_3\}
$$
satisfies that $[X_1,X_2]=[X_1,X_3]=[X_2,X_3]=0$. Specifically, by using
$[E_i,E_j]=\nabla_{E_i} {E_j}-\nabla_{E_j} {E_i}$, \eqref{eqn:2.25} and $b_2=0$, we get
$$
[E_1,E_2]=0,\ [E_1,E_3]=-b_1\sqrt{\tfrac{1-C}{1+C}}E_1,\ [E_2,E_3]=b_4\sqrt{\tfrac{1+C}{1-C}}E_2.
$$
Therefore
$$
\begin{aligned}
{\text [X_1,X_3]}&=-(\sqrt{\tfrac{1-C}{1+C}}\rho_1b_1+E_3\rho_1)E_1.
\end{aligned}
$$
Then $[X_1,X_3]=0$ is equivalent to
\begin{equation}\label{eqn:m1}
E_3\rho_1=-\sqrt{\tfrac{1-C}{1+C}}\rho_1b_1.
\end{equation}
Similarly, we get
$$
{\text [X_2,X_3]}=(\sqrt{\tfrac{1+C}{1-C}}\rho_2b_4-E_3\rho_2)E_2,\ \
{\text [X_1,X_2]}=\rho_1(E_1\rho_2)E_2-\rho_2(E_2\rho_1)E_1.
$$
Then $[X_2,X_3]=[X_1,X_2]=0$ is  equivalent to
\begin{equation}\label{eqn:m2}
E_3\rho_2=\sqrt{\tfrac{1+C}{1-C}}\rho_2b_4,
\end{equation}
\begin{equation}\label{eqn:dm1dm4}
\rho_1 E_1\rho_2=\rho_2E_2\rho_1=0.
\end{equation}
Then by using \eqref{eqn:b1b4}, we can integrate the equations \eqref{eqn:m1}
and \eqref{eqn:m2} along the integral curves of $E_3$ directly and  obtain that
\begin{equation}\label{eqn:m1m2}
\begin{aligned}
&\rho_1=h_3\cos\big(\sqrt{\tfrac{1-C}{2}}(t+h_1)\big),\\
&\rho_2=h_4\cos\big(\sqrt{\tfrac{1+C}{2}}(t+h_2)\big),
\end{aligned}
\end{equation}
where $h_3$ and $h_4$ are functions on $M$ which satisfy $E_3h_3=E_3h_4=0$.
In particular, we chose $h_3=h_4=1$, and then $\rho_1=\cos\big(\sqrt{\tfrac{1-C}{2}}(t+h_1)\big)$,
$\rho_2=\cos\big(\sqrt{\tfrac{1+C}{2}}(t+h_2)\big)$.

On the other hand, from \eqref{eqn:b1b4}, we have
$$
t+h_1=\sqrt{\tfrac{2}{1-C}}\arctan(\sqrt{\tfrac{2}{1+C}}b_1),\ \
t+h_2=\sqrt{\tfrac{2}{1+C}}\arctan(-\sqrt{\tfrac{2}{1-C}}b_4).
$$
Combining these equations with $\rho_1=\cos\big(\sqrt{\tfrac{1-C}{2}}(t+h_1)\big)$,
$\rho_2=\cos\big(\sqrt{\tfrac{1+C}{2}}(t+h_2)\big)$, we have
$$
\rho_1=\cos\big(\arctan(\sqrt{\tfrac{2}{1+C}}b_1)\big),\ \
\rho_2=\cos\big(\arctan(-\sqrt{\tfrac{2}{1-C}}b_4)\big).
$$
From \eqref{eqn:2.29}, \eqref{eqn:2.30} and $b_2=0$,
it follows that $E_1\rho_2=E_2\rho_1=0$, which implies that $[X_1,X_2]=[X_1,X_3]=[X_2,X_3]=0$.



By the definition of the frame $\{X_1,X_2,X_3\}$ and using $[X_1,X_2]=[X_1,X_3]=[X_2,X_3]=0$,
we can identify $M$ with an open subset $\Omega$ of $\mathbb{R}^3$ locally and express the hypersurface $M$ by an immersion 
$$
\begin{aligned}
\Phi:\Omega
\longrightarrow\mathbb{S}^2\times\mathbb{S}^2,\ \ \
(t,r,s)\mapsto (p(t,r,s),q(t,r,s)),
\end{aligned}
$$
such that $(\frac{\partial p}{\partial t},\frac{\partial q}{\partial t})=E_3$, and
$(\frac{\partial p}{\partial r},\frac{\partial q}{\partial r}),
(\frac{\partial p}{\partial s},\frac{\partial q}{\partial s})$
are collinear with $E_1,E_2$, respectively.
By the definition of $P$  and using
$$
PE_1=E_1,\ \ PE_2=-E_2,
$$
we obtain that $dp,dq:T(\Omega)\rightarrow
T\mathbb{S}^2$ satisfy
\begin{equation}\label{eqn:5.49}
\left\{
\begin{aligned}
(dp(\tfrac{\partial}{\partial r}),0)&=\tfrac{1}{2}(d\Phi(\tfrac{\partial}{\partial r})+Pd\Phi(\tfrac{\partial}{\partial r}))=d\Phi(\tfrac{\partial}{\partial r}),\\
(0,dq(\tfrac{\partial}{\partial r}))&=\tfrac{1}{2}(d\Phi(\tfrac{\partial}{\partial r})-Pd\Phi(\tfrac{\partial}{\partial r}))=0,
\end{aligned}
\right.
\end{equation}
%
%
and
\begin{equation}\label{eqn:5.50}
\left\{
\begin{aligned}
(dp(\tfrac{\partial}{\partial s}),0)&=\tfrac{1}{2}(d\Phi(\tfrac{\partial}{\partial s})+Pd\Phi(\tfrac{\partial}{\partial s}))=0,\\
(0,dq(\tfrac{\partial}{\partial s}))&=\tfrac{1}{2}(d\Phi(\tfrac{\partial}{\partial s})-Pd\Phi(\tfrac{\partial}{\partial s}))=d\Phi(\tfrac{\partial}{\partial s}).
\end{aligned}
\right.
\end{equation}

The first equation of \eqref{eqn:5.50} shows that $p$ depends only on the $(t,r)$.
Similarly, from the second equation in \eqref{eqn:5.49} we derive that $q$ depends only on
the $(t,s)$. By $g(PE_3,E_3)=-C$, $AE_3=0$ and $\nabla_{E_3}{E_3}=0$, we have that
integral curves of $E_3$ are geodesics of $\mathbb{S}^2\times\mathbb{S}^2$, and we can assume that
%
\begin{equation*}
\begin{aligned}
&p(t,r)=\cos(\sqrt{\tfrac{1-C}{2}}t+c_1(r))\gamma(r)+\sin(\sqrt{\tfrac{1-C}{2}}t+c_1(r))
(m_1(r)\frac{d\gamma(r)}{dr}+m_2(r)N(r)),\\
&q(t,s)=\cos(\sqrt{\tfrac{1+C}{2}}t+c_2(s))\tilde{\gamma}(s)+\sin(\sqrt{\tfrac{1+C}{2}}t+c_2(s))
(n_1(s)\frac{d\tilde{\gamma}(s)}{ds}+n_2(s)\tilde{N}(s)),
\end{aligned}
\end{equation*}
where $\gamma(r),\tilde{\gamma}(s)\in\mathbb{S}^2$,
$N(r)$ and $\tilde{N}(s)$ are unit normal vector fields of $\gamma(r)$ and $\tilde{\gamma}(s)$
in $\mathbb{S}^2$, respectively.
We can re-parametrize the curves to use  arc length parameters (which we still denote by $r$ and $s$), then
$|\frac{d\gamma(r)}{dr}|=|\frac{d\tilde{\gamma}(s)}{ds}|=m_1^2(r)+m_2^2(r)=n_1^2(s)+n_2^2(s)=1$.
Assume that $k(r)$ and $\tilde{k}(s)$ are the curvatures of $\gamma(r)$
and $\tilde{\gamma}(s)$, respectively, then we have
\begin{equation}\label{eq:qx}
\begin{aligned}
\frac{d^2\gamma(r)}{dr^2}&=-\gamma(r)+k(r)N(r), \ \
\frac{dN(r)}{dr}=-k(r)\frac{d\gamma(r)}{dr},\\
\frac{d^2\tilde{\gamma}(s)}{ds^2}&=-\tilde{\gamma}(s)+\tilde{k}(s)\tilde{N}(s), \ \
\frac{d\tilde{N}(s)}{ds}=-\tilde{k}(s)\frac{d\tilde{\gamma}(s)}{ds}.
\end{aligned}
\end{equation}
By using the fact $g(E_1,E_3)=g(E_2,E_3)=0$, we have
$$
g((\frac{\partial p}{\partial t},\frac{\partial q}{\partial t}),
(\frac{\partial p}{\partial r},\frac{\partial q}{\partial r}))=
g((\frac{\partial p}{\partial t},\frac{\partial q}{\partial t}),
(\frac{\partial p}{\partial s},\frac{\partial q}{\partial s}))=0,
$$
which implies that
\begin{equation}\label{eq:6.4}
\frac{dc_1(r)}{dr}+m_1(r)=\frac{dc_2(s)}{ds}+n_1(s)=0.
\end{equation}

Finally, we take a transformation to simplify the expression of $(p,q)$ as follows.
\begin{equation*}
\begin{aligned}
p(t,r)&=\cos(\sqrt{\tfrac{1-C}{2}}t+c_1(r))\gamma(r)+\sin(\sqrt{\tfrac{1-C}{2}}t+c_1(r))
(m_1(r)\frac{d\gamma(r)}{dr}+m_2(r)N(r))\\
&=\cos(\sqrt{\tfrac{1-C}{2}}t)V_1(r)+\sin(\sqrt{\tfrac{1-C}{2}}t)V_2(r),
\end{aligned}
\end{equation*}
where
$$
\begin{aligned}
V_1(r)&=\cos(c_1(r))\gamma(r)+\sin(c_1(r))(m_1(r)\frac{d\gamma(r)}{dr}+m_2(r)N(r)),\\
V_2(r)&=-\sin(c_1(r))\gamma(r)+\cos(c_1(r))(m_1(r)\frac{d\gamma(r)}{dr}+m_2(r)N(r)).
\end{aligned}
$$
By using \eqref{eq:6.4}, it can be checked directly that $|V_1(r)|=|V_2(r)|=1$ and
$\langle V_1(r),V_2(r)\rangle=\langle \frac{dV_1(r)}{dr},V_2(r)\rangle=0$,
i.e., $V_2(r)$ is the unit normal vector field of $V_1(r)\hookrightarrow \mathbb{S}^2$.
For $q(t,s)$, we can take a similar transformation.
Furthermore, by taking re-parameterizations to make $r$ and $s$ being arc length parameters,
$M$ is locally given by map $\Phi: \Omega\subset\mathbb{R}^{3}\longrightarrow \mathbb{S}^2\times\mathbb{S}^2$:
\ $(t,r,s)\longrightarrow(p,q)$ defined by \eqref{eqn:6.1}.
Conversely, one can verify directly  that the
hypersurfaces constructed by \eqref{eqn:6.1} have constant $C$ and satisfy that $b_2=0$.
\end{proof}

Armed with Lemma \ref{lemma:2.1aaa1}, Theorem \ref{thm:3.1} and Proposition \ref{prop:3.61},
we  obtain the following classification result for
 minimal hypersurface of $\mathbb{S}^2\times \mathbb{S}^2$ with constant $C=0$.
\begin{theorem}\label{thm:3.2aa}
Let $M$ be a minimal hypersurface of $\mathbb{S}^2\times \mathbb{S}^2$ with constant $C=0$.
Then, either
\begin{enumerate}
\item[(1)]
$M$ is an open part of $M_0$; or

\item[(2)]
$M$ is an open part of $M_{a,b}$ for some $a,b\in\mathbb{S}^2$; or

\item[(3)]
$M$ is an open part of a minimal hypersurface described by Theorem \ref{thm:3.1}.
\end{enumerate}
\end{theorem}
\begin{proof}
By Lemma \ref{lemma:2.1aaa1}, under the frame $\{E_1,E_2,E_3\}$, there are three cases
of $b_1,b_2,b_4$ on $M$, and  Case 3 has been classified by Theorem \ref{thm:3.1}.

For Case 1, it is not hard to see that $M$ has constant mean curvature and
constant scalar curvature. Then, by Corollary 1 (3) of \cite{Ur} and the fact that $M_0$
is the only minimal hypersurface among $M_t$, we know that  Case 1 corresponds to the
hypersurface $M_0=\{(p,q)\in\mathbb{S}^2\times \mathbb{S}^2|\ \langle p,q\rangle=0\}$.

For Case 2, by Proposition \ref{prop:3.61}, $M$ is locally given by the immersion
$\Phi: \Omega\subset\mathbb{R}^{3}\longrightarrow \mathbb{S}^2\times\mathbb{S}^2$:
\ $(t,r,s)\longrightarrow(p,q)$ defined by \eqref{eqn:6.1} with $C=0$. Hence, we only need to
determine the specific curves $\gamma(r)$ and $\tilde{\gamma}(s)$.
We assume as in the proof of Proposition \ref{prop:3.61} that $\kappa(r)$ and $\tilde{\kappa}(s)$ are the  curvatures of $\gamma(r)$ and
$\tilde{\gamma}(s)$, respectively. Then, \eqref{eq:qx} hold on $\gamma(r)$ and
$\tilde{\gamma}(s)$.
By direct calculations, we have
$$E_1=(\frac{d\gamma(r)}{d r},0),\ E_2=(0,\frac{d\tilde{\gamma}(s)}{d s}),
\ E_3=(\frac{\partial p}{\partial t},\frac{\partial q}{\partial t}),
\ N=(\frac{\partial p}{\partial t},-\frac{\partial q}{\partial t}),$$
and
$$
\begin{aligned}
b_1&=g(AE_1,E_1)=\tfrac{1}{\sqrt{2}}\frac{\sin(\frac{t}{\sqrt{2}})
+\cos(\frac{t}{\sqrt{2}})\kappa(r)}{\cos(\frac{t}{\sqrt{2}})
-\sin(\frac{t}{\sqrt{2}})\kappa(r)}, \ \ b_2=g(AE_1,E_2)=0,\\
b_4&=g(AE_2,E_2)=-\tfrac{1}{\sqrt{2}}\frac{\sin(\frac{t}{\sqrt{2}})
+\cos(\frac{t}{\sqrt{2}})\tilde{\kappa}(s)}{\cos(\frac{t}{\sqrt{2}})
-\sin(\frac{t}{\sqrt{2}})\tilde{\kappa}(s)}.
\end{aligned}
$$
From $H=\frac{1}{3}(b_1+b_4)=0$, we have that $\kappa(r)$ and $\tilde{\kappa}(s)$
are equal to the same constant. Up to an isometry of $\mathbb{S}^2\times \mathbb{S}^2$,
we assume that $\gamma(r)$ and $\tilde{\gamma}(s)$ are both
$\mathbb{S}^1(r_1)\hookrightarrow \mathbb{S}^2$ such that the third coordinate component
is $\sqrt{1-r_1^2}$, then $M$ satisfies that
$\langle p,(0,0,1)\rangle+\langle q,(0,0,-1)\rangle=0$ for any $(p,q)\in M$.
Therefore, in general,  $M$ is an open part of $M_{a,b}$ for some $a,b\in\mathbb{S}^2$.
\end{proof}

\section{Proofs of Theorem \ref{thm:1.1} and Theorem \ref{thm:1.2}}\label{sect:4}~

In this section, we prove  Theorem \ref{thm:1.1} and Theorem \ref{thm:1.2}.
As the first key step, by using the so-called Tsinghua principle, which was first discovered by
the first three authors in 2013 at Tsinghua University \cite{A-L-V-W}, we establish a very useful identity
on any hypersurface $M$ of $\mathbb{S}^2\times\mathbb{S}^2$ with constant sectional curvature.

\begin{lemma}\label{lemma:3.2aa}
Let $M$ be a hypersurface of $\mathbb{S}^2\times \mathbb{S}^2$ with constant sectional curvature  $\kappa$.
Then, for any tangent vector fields $W,U,Y,Z\in TM$, the following equation holds:
\begin{equation}\label{eqn:3.1}
\mathop\mathfrak{S}\limits_{W,U,Y}\mathbf{I}(W,U,Y,Z)=0,
\end{equation}
where the symbol $\mathfrak{S}$ stands for the cyclic summation,
and $\mathbf{I}(W,U,Y,Z)$ is defined by
\begin{equation*}
\mathbf{I}(W,U,Y,Z):=\tfrac{1}{2}\{-g(TY,Z)g(AW,TU)+g(TU,Z)g(AW,TY)\}.
\end{equation*}
\end{lemma}
\begin{proof}

In order to prove Lemma \ref{lemma:3.2aa}, we calculate the expression of the cyclic
summation
\begin{equation}\label{eqn:3.2}
\mathfrak{A}:=\mathop\mathfrak{S}\limits_{W,U,Y} \big\{g((\nabla^2 A)(W,U,Y),Z)-
g((\nabla^2 A)(W,Y,U),Z)\big\}
\end{equation}
in two different ways.
Specifically, on the one hand, we take the covariant derivative
of the Codazzi equation \eqref{eqn:2.2} and obtain that
\begin{equation*}
\begin{split}
&g((\nabla^2 A)(W,U,Y),Z)-g((\nabla^2 A)(W,Y,U),Z)\\
&=\tfrac{1}{2}\{g((\nabla_W T)Y,Z)\mu(U)+g(TY,Z)(\nabla_W \mu)U\\
&\ \ -g((\nabla_W T)U,Z)\mu(Y)-g(TU,Z)(\nabla_W \mu)Y\}.
\end{split}
\end{equation*}
By direct calculations, with the use of Lemma \ref{lemma:2.1},
we have
\begin{equation*}
\begin{split}
g&((\nabla^2 A)(W,U,Y),Z)-g((\nabla^2 A)(W,Y,U),Z)\\
&=\tfrac{1}{2}\{g(AW,Y)\mu(Z)\mu(U)-g(TY,Z)g(AW,TU)+Cg(TY,Z)g(AW,U)\\
&\ \ -g(AW,U)\mu(Z)\mu(Y)+g(TU,Z)g(AW,TY)-Cg(TU,Z)g(AW,Y)\}.
\end{split}
\end{equation*}
Thus, it can be checked directly that
\begin{equation}\label{eqn:3.3aa}
\mathfrak{A}=\mathop\mathfrak{S}\limits_{W,U,Y}\mathbf{I}(W,U,Y,Z).
\end{equation}

On the other hand, the cyclic sum $\mathfrak{A}$ can be rewritten as
\begin{equation}\label{eqn:3.3}
\mathfrak{A}=\mathop\mathfrak{S}\limits_{W,U,Y}\big\{g((\nabla^{2}A)(W,U,Y),Z)
-g((\nabla^{2}A)(U,W,Y),Z)\big\}.
\end{equation}
Then we can apply the Ricci identity \eqref{eqn:ric} and obtain that
\begin{equation}\label{eqn:3.4}
\mathfrak{A}=-\mathop\mathfrak{S}\limits_{W,U,Y} \big\{g(R(W,U)Y,AZ)
+g(R(W,U)Z,AY)\big\}.
\end{equation}
Since $M$ has constant sectional curvature $\kappa$
of $\mathbb{S}^2\times \mathbb{S}^2$, we have
\begin{equation}\label{eqn:3.6aa}
g(R(W,U)Y,Z)=\kappa\{g(U,Y)g(W,Z)-g(W,Y)g(U,Z)\}.
\end{equation}
It follows directly  that
$$
\mathfrak{A}=-\mathop\mathfrak{S}\limits_{W,U,Y} \big\{g(R(W,U)Y,AZ)
+g(R(W,U)Z,AY)\big\}=0,
$$
which combined with  \eqref{eqn:3.3aa} leads to the conclusion of Lemma \ref{lemma:3.2aa}.
\end{proof}

\begin{remark}\label{rem:4.1aaaa}
The method used in the proof of Lemma \ref{lemma:3.2aa}  is
called the Tsinghua principle. This remarkable technique has been applied in many
different situations since its first successful attempt in \cite{A-L-V-W},
see \cite{CHMV-1,D-V-W,LXX,LMVVW,V-W,Y-Y-H} for details.
\end{remark}

\begin{proposition}\label{prop:4.5}
Let $M$ be a hypersurface of $\mathbb{S}^2\times\mathbb{S}^2$ with constant sectional curvature $\kappa$.
Then $\kappa=\tfrac{1}{2}$ and the product angle function $C=0$.
\end{proposition}
\begin{proof}
Assume that hypersurface $M$ has constant sectional curvature $\kappa$.
Since our concern is only local, in order to prove this proposition,
we will divide the discussions into two cases depending on the value of
the product angle function $C$ on $M$.

{\bf Case I}. $C\neq\pm1$ on $M$.

In this case, we can take the local orthonormal frame fields $\{E_1,E_2,E_3\}$
defined by \eqref{eqn:2.5}.
Then, the shape operator $A$, $(1,1)$-tensor $T$ and $1$-form $\mu$
are given by \eqref{eqn:2.6} and \eqref{eqn:2.7}.
By using $\bar{\nabla}P=0$, \eqref{eqn:2.6}, Lemma \ref{lemma:2.1},
we obtain that
\begin{equation}\label{eqn:LL}
\left\{
\begin{aligned}
&E_1C=-2b_3\sqrt{1-C^2},\ E_2C=-2b_5\sqrt{1-C^2},\ E_3C=-2b_6\sqrt{1-C^2},\\ 		
&\nabla_{E_1}E_1=b_1\sqrt{\tfrac{1-C}{1+C}}E_3,\
\nabla_{E_1}E_2=-b_2\sqrt{\tfrac{1+C}{1-C}}E_3,\ \nabla_{E_1}E_3=-b_1\sqrt{\tfrac{1-C}{1+C}}E_1
+b_2\sqrt{\tfrac{1+C}{1-C}}E_2,\\
&\nabla_{E_2} E_1=b_2\sqrt{\tfrac{1-C}{1+C}}E_3,\
\nabla_{E_2} E_2=-b_4\sqrt{\tfrac{1+C}{1-C}}E_3,\ \nabla_{E_2}E_3=-b_2\sqrt{\tfrac{1-C}{1+C}}E_1
+b_4\sqrt{\tfrac{1+C}{1-C}}E_2,\\
&\nabla_{E_3}E_1=b_3\sqrt{\tfrac{1-C}{1+C}}E_3,\
\nabla_{E_3}E_2=-b_5\sqrt{\tfrac{1+C}{1-C}}E_3,\ \nabla_{E_3}E_3=-b_3\sqrt{\tfrac{1-C}{1+C}}E_1
+b_5\sqrt{\tfrac{1+C}{1-C}}E_2.
\end{aligned}\right.
\end{equation}
Taking in \eqref{eqn:3.1},  $
(W,U,Y,Z)=(E_1,E_2,E_3,E_1), \ (E_1,E_2,E_3,E_2), \ (E_1,E_2,E_3,E_3),
$ respectively, 
we obtain that $b_3(1+C)=b_5(1-C)=b_2C=0$.
Since $C\neq\pm1$, we have $b_3=b_5=0$.

\vskip2mm
{\bf Claim.\ \ $C=0$ and $\kappa=\tfrac{1}{2}$ on $M$.}

{\it Proof of the Claim:}\ \  We use reduction to absurdity. If $C\neq0$ holds in an open subset of $M$,
then $b_2=0$. By direct calculations, we have
$$
g(R(E_1,E_2)E_2,E_1)=b_1b_4,\
g(R(E_1,E_3)E_3,E_1)=\tfrac{1}{2}-\tfrac{C}{2}+b_1b_6,\
g(R(E_2,E_3)E_3,E_2)=\tfrac{1}{2}+\tfrac{C}{2}+b_4b_6.
$$
Since $C\neq 0$, we have $b_6\neq 0$. As $C\neq\pm1$, it follows that $$\kappa\neq0,~b_6=\tfrac{2\kappa-1+C}{2b_1}=\tfrac{2\kappa-1-C}{2b_4},~b_6^2=\tfrac{(2\kappa-1)^2-C^2}{4\kappa}.$$
If $\kappa<0$, then $(2\kappa-1)^2-C^2>0$,
which contradicts $b_6^2=\tfrac{(2\kappa-1)^2-C^2}{4\kappa}>0$. It follows that $\kappa>0$.
Without loss of generality, we assume that
\begin{equation}\label{eqn:6.2}
b_6=-\tfrac{\sqrt{(2\kappa-1)^2-C^2}}{2\sqrt{\kappa}},\
b_1=-\tfrac{\sqrt{\kappa}(2\kappa-1+C)}{\sqrt{(2\kappa-1)^2-C^2}},
\ b_4=-\tfrac{\sqrt{\kappa}(2\kappa-1-C)}{\sqrt{(2\kappa-1)^2-C^2}}.
\end{equation}

Next, taking $(Y,Z)=(E_3,E_1)$ and $(Y,Z)=(E_3,E_2)$ into Codazzi equation \eqref{eqn:2.2},
with the use of \eqref{eqn:2.6}, \eqref{eqn:LL},
we compare the components of $E_1$ and $E_2$, respectively. It follows that
$$
\begin{aligned}
E_3(b_1)-E_1(b_3)&=\tfrac{\sqrt{1-C^2}}{2}
+2b_3^{2}\sqrt{\tfrac{1-C}{1+C}}-b_6b_1\sqrt{\tfrac{1-C}{1+C}}
+b_1^{2}\sqrt{\tfrac{1-C}{1+C}}-b_2^{2}\sqrt{\tfrac{1+C}{1-C}},\\
E_3(b_4)-E_2(b_5)&=-\tfrac{\sqrt{1-C^2}}{2}-2b_5^{2}\sqrt{\tfrac{1+C}{1-C}}
+b_6b_4\sqrt{\tfrac{1+C}{1-C}}+b_2^{2}\sqrt{\tfrac{1-C}{1+C}}
-b_4^{2}\sqrt{\tfrac{1+C}{1-C}}.
\end{aligned}
$$
By using $b_2=b_3=b_5=0$, the above two equations become
\begin{equation}\label{eqn:6.3}
E_3b_1=\tfrac{\sqrt{1-C}(1+C+2b_1(b_1-b_6))}{2\sqrt{1+C}},\
E_3b_4=\tfrac{\sqrt{1+C}(-1+C+2b_4(b_6-b_4))}{2\sqrt{1-C}}.
\end{equation}
In the following, we take the derivatives of $b_1$ and $b_4$ with respect to $E_3$,
with the use of \eqref{eqn:6.2} and $E_3C=-2b_6\sqrt{1-C^2}$, we obtain that
\begin{equation}\label{eqn:6.4}
E_3b_1=\tfrac{\sqrt{1-C^2}(2\kappa-1)}{1+C-2\kappa},\
E_3b_4=\tfrac{\sqrt{1-C^2}(2\kappa-1)}{-1+C+2\kappa}.
\end{equation}
By \eqref{eqn:6.3} and \eqref{eqn:6.4}, we have $\sqrt{1-C^2}(2\kappa-1)=0$,
from which we get $\kappa=\tfrac{1}{2}$.
However, this contradicts $C\neq0$ and $\sqrt{(2\kappa-1)^2-C^2}\geq0$.
Therefore, the product angle fucntion $C=0$.

Finally, by Lemma \ref{lemma:2.1}, we have $b_3=b_5=b_6=0$.
By direct calculations, we have
$$
g(R(E_1,E_2)E_2,E_1)=b_1b_4-b_2^2, \ \ g(R(E_1,E_3)E_3,E_1)=g(R(E_2,E_3)E_3,E_2)=\tfrac{1}{2},
$$
which implies  that the sectional curvature of $M$ is $\kappa=\tfrac{1}{2}$
and $b_1b_4-b_2^2=\tfrac{1}{2}$ holds on $M$. We have completed the proof of above Claim.

{\bf Case II}. $C^2=1$ on $M$.

In this case, by Lemma \ref{lem:2.2}, we know that $M$ is an open part of
$\Gamma\times \mathbb{S}^2$, where $\Gamma$
is a curve of $\mathbb{S}^2$. It follows that the tangent bundle of $M$ has
a decomposition $TM=T\Gamma\oplus T\mathbb{S}^2$
and we have $A|_{T\mathbb{S}^2}=0$.
Next, we take an  orthonormal basic $\{v_1,v_2,v_3\}$ with $v_1\in T\Gamma$ and $v_2,v_3\in T\mathbb{S}^2$,
after  direct calculations, we have $g(R(v_1,v_2)v_2,v_1)=0$ and $g(R(v_2,v_3)v_3,v_2)=1$.
Therfore, in this case, there exists no hypersurface with constant sectional curvature.
\end{proof}

\begin{remark}\label{rem:4.1aa}
As a direct consequence of Proposition \ref{prop:4.5},  we obtain that there exists no hypersurface in $\mathbb{S}^2\times\mathbb{S}^2$ with constant sectional curvature  and constant mean curvature.
In fact, let $M$ be a hypersurface of $\mathbb{S}^2\times\mathbb{S}^2$ with constant sectional curvature $\kappa$, by Proposition \ref{prop:4.5},  $\kappa=\frac{1}{2}$ and $M$ has constant $C=0$. According to Corollary 1 (3) of \cite{Ur},
$\{M_t,~t\in(-1,1)\}$ are the only hypersurfaces of $\mathbb{S}^2\times\mathbb{S}^2$ which have the mean curvature,
the scalar curvature and the function $C$ constants. However, for any $t\in (-1,1)$, $M_t$ does not have
constant sectional curvature.
\end{remark}

Before we prove Theorem  \ref{thm:1.1} and Theorem \ref{thm:1.2}, we need some preparations.
For any hypersurface $M$ of $\mathbb{S}^2\times\mathbb{S}^2$ with product angle function $C\neq\pm1$ in local,
let $\{E_1, E_2, E_3\}$ be the orthonormal frame on $M$ defined in \eqref{eqn:2.5}. We define the map
$$
\Phi_r: M\rightarrow \mathbb{S}^2\times\mathbb{S}^2,\ \ \ \ \ (p,q) \mapsto \Phi_r((p,q))=\tilde{\exp}_{(p,q)}(rN_{(p,q)}), \quad \Phi_{0}(M)=M,
$$
where $(p,q)\in M$ and ${\rm \tilde{exp}}$ is the Riemannian exponential map of $\mathbb{S}^2\times\mathbb{S}^2$.
For any $(p,q)\in M$, denote $\gamma(t)=\tilde{\exp}_{(p,q)}(tN_{(p,q)})$ the
geodesic satisfying $\gamma(0)=(p,q)$ and $\gamma'(0)=N_{(p,q)}$. Let $\{E_1^r, E_2^r, E_3^r\}$
be the parallel translation of $\{E_1, E_2, E_3\}$ along $\gamma(t)$ at point $\Phi_r((p,q))$, it follows that $\{E_1^r, E_2^r, E_3^r\}$ forms an orthonormal basis of
$(T\mathbb{S}^2\times\mathbb{S}^2)\setminus\{\gamma'(r)\}$ at $\Phi_r((p,q))$.

For a fixed $r$, $\Phi_r(M)$ is not necessarily a submanifold of $\mathbb{S}^2\times\mathbb{S}^2$,
but at least locally and for $r$ small enough, it is a hypersurface of $\mathbb{S}^2\times\mathbb{S}^2$.
$\Phi_r(M)$ is a hypersurface if and only if the tangent map $(\Phi_r((p,q)))_*$ is
non-degenerate for any $(p,q)\in M$. In this case, the corresponding frame $\{E_1^r, E_2^r, E_3^r\}$ forms an orthonormal basis of $T_{\Phi_r((p,q))}\Phi_r(M)$.

In the following, for any hypersurface $M$ and any $r\in \mathbb{R}$, we give some detailed geometric properties of the map $\Phi_r$ when $M$ has constant product angle function $C=0$, by applying the theory of Jacobi fields. More precisely,  we determine the tangent map of $\Phi_r$, and when  $\Phi_r(M)$
is a hypersurface,  we determine the shape operator of  $\Phi_r(M)$. There properties will be used in the proofs of Theorem  \ref{thm:4.6} and Theorem \ref{thm:1.1}.

\begin{proposition}\label{prop:4.3}
Let $M$ be a hypersurface of $\mathbb{S}^2\times\mathbb{S}^2$ with $C=0$,
and $\{E_1, E_2, E_3\}$ be the frame field on $M$ defined by \eqref{eqn:2.5}.
Then, for any $(p,q)\in M$, the tangent map of $\Phi_r$ has the following expression:
\begin{equation}\label{eqn:rank}
\left(
  \begin{array}{c}
    (\Phi_r)_*{E_1} \\
    (\Phi_r)_*{E_2} \\
    (\Phi_r)_*{E_3} \\
  \end{array}
\right)=
(B_{ij})\left(
  \begin{array}{c}
    E_1^r \\
    E_2^r \\
    E_3^r \\
  \end{array}
\right),
\end{equation}
where
\begin{equation}\label{Bij}
(B_{ij})=
\left(
\begin{array}{ccc}
\cos(\frac{r}{\sqrt{2}})-\sqrt{2}b_1\sin(\frac{r}{\sqrt{2}}) & -\sqrt{2}b_2\sin(\frac{r}{\sqrt{2}}) & 0 \\
-\sqrt{2}b_2\sin(\frac{r}{\sqrt{2}}) & \cos(\frac{r}{\sqrt{2}})
-\sqrt{2}b_4\sin(\frac{r}{\sqrt{2}}) & 0 \\
0 & 0 & 1 \\
\end{array}
\right).
\end{equation}
The determinant of $(B_{ij})$ is given by
\begin{equation}\label{eqn:det}
{\rm Det}(B_{ij})=\tfrac{1}{2}\Big(1-2b_2^2+2b_1b_4+(1+2b_2^2-2b_1b_4)\cos(\sqrt{2}r)
-\sqrt{2}(b_1+b_4)\sin(\sqrt{2}r)\Big).
\end{equation}

Furthermore, when $\Phi_r(M)$ is a hypersurface,  which is equivalent to ${\rm Det}(B_{ij})\neq 0$.
Let $A_r$ be the  shape operator of $\Phi_r(M)$, then
for any $(p,q)\in M$, the expression of  $A_r$ at $\Phi_r((p,q))$ is given by
\begin{equation}\label{eqn:Arr}
\left(
  \begin{array}{c}
    A_rE_1^r \\
    A_rE_2^r \\
    A_rE_3^r \\
  \end{array}
\right)=
\left(
\begin{array}{ccc}
(A_r)_{11}  &
(A_r)_{12}  & 0 \\
(A_r)_{21} &
(A_r)_{22} & 0 \\
0 & 0 & 0 \\
\end{array}
\right)
\left(
  \begin{array}{c}
    E_1^r \\
    E_2^r \\
    E_3^r \\
  \end{array}
\right),
\end{equation}
where
$$
\begin{aligned}
(A_r)_{11}&=\frac{\sqrt{2}(b_1-b_4+(b_1+b_4)\cos(\sqrt{2}r))
+(1+2b_2^2-2b_1b_4)\sin(\sqrt{2}r)}
{\sqrt{2}\Big(1-2b_2^2+2b_1b_4+(1+2b_2^2-2b_1b_4)\cos(\sqrt{2}r)
-\sqrt{2}(b_1+b_4)\sin(\sqrt{2}r)\Big)},\\
(A_r)_{12}&=(A_r)_{21}=\frac{2b_2}{1-2b_2^2+2b_1b_4+(1+2b_2^2-2b_1b_4)\cos(\sqrt{2}r)
-\sqrt{2}(b_1+b_4)\sin(\sqrt{2}r)},\\
(A_r)_{22}&=\frac{\sqrt{2}(-b_1+b_4+(b_1+b_4)\cos(\sqrt{2}r))
+(1+2b_2^2-2b_1b_4)\sin(\sqrt{2}r)}
{\sqrt{2}\Big(1-2b_2^2+2b_1b_4+(1+2b_2^2-2b_1b_4)\cos(\sqrt{2}r)
-\sqrt{2}(b_1+b_4)\sin(\sqrt{2}r)\Big)}.
\end{aligned}
$$

\end{proposition}
\begin{proof}
We choose the orthonormal frame fields $\{E_1,E_2,E_3\}$ defined by \eqref{eqn:2.5}
and assume that \eqref{eqn:2.7} holds. Since $M$ has constant $C=0$, it follows that $b_3=b_5=b_6=0$.

Let $N=(N_1,N_2)$ be the unit normal field to $M$, then for any given $r\in \mathbb{R}$,
we define a map $\Phi_r:M\rightarrow\mathbb{S}^2\times\mathbb{S}^2$ by
$$
\Phi_r( (p,q))=\tilde{\exp}_{(p,q)}(rN_{(p,q)})=({\rm exp}_prN_1,{\rm exp}_qrN_2), \ \ (p,q)\in M,
$$
where exp denotes the exponential map in $\mathbb{S}^2$.
As $C=0$, we have $|N_1|=|N_2|=\tfrac{1}{\sqrt{2}}$, then
$$
\Phi_r(M)=\big(\cos(\tfrac{r}{\sqrt{2}})p+\sqrt{2}\sin(\tfrac{r}{\sqrt{2}})N_1,
\cos(\tfrac{r}{\sqrt{2}})q+\sqrt{2}\sin(\tfrac{r}{\sqrt{2}})N_2\big), \ \ (p,q)\in M.
$$
For any $(p,q)\in M$, let $\gamma(t)$ be the geodesic such that $\gamma(0)=(p,q)$
and $\gamma'(0)=N_{(p,q)}$. Along $\gamma(t)$, by using $\bar{\nabla}P=0$ and
$\bar{\nabla}_{\gamma'(t)} {\gamma'(t)}=0$, we have
$$
\frac{d}{d t}g(P\gamma'(t),\gamma'(t))=g(\bar{\nabla}_{\gamma'(t)} P\gamma'(t),\gamma'(t))=0.
$$
It follows that
\begin{equation}\label{eqn:CC}
g(P\gamma'(r),\gamma'(r))=g(P\gamma'(0),\gamma'(0))=C=0.
\end{equation}

Let
\begin{equation}\label{eqn:EE}
E_1^t=\frac{J_1\gamma'(t)+J_2\gamma'(t)}{\sqrt{2}}, \
E_2^t=\frac{J_1\gamma'(t)-J_2\gamma'(t)}{\sqrt{2}}, \ E_3^t=P\gamma'(t),
\end{equation}
then $\{E_1^t,E_2^t,E_3^t\}$ forms three orthonormal parallel vector fields along $\gamma(t)$,
and $E_1^0=E_1,\ E_2^0=E_2,\ E_3^0=E_3$. Furthermore, along $\gamma(t)$, it holds that
\begin{equation}\label{eqn:PP}
P E_1^t=E_1^t,\ \ P E_2^t=- E_2^t,\ \ P E_3^t=\gamma'(t),\ \ P\gamma'(t)=E_3^t.
\end{equation}
Recall that as described in Section 10.2.1 of \cite{B-C-O}, a Jacobi field $Y(t)$ along $\gamma(t)$ is called an $M$-Jacobi field if its initial values satisfy the following two conditions:
$$
Y(0)\in T_{\gamma(0)}M \ \ {\rm and} \ \ Y'(0)+AY(0)\in T^\bot_{\gamma(0)}M.
$$
We define $\mathfrak{J}(M,\gamma(t))$ as the $3$-dimensional vector
space consisting of all $M$-Jacobi vector fields along $\gamma(t)$ which are perpendicular
to the $M$-Jacobi vector field $t\rightarrow t\gamma'(t)$. 
Then, at $\gamma(r)\in\Phi_r(M)$, the tangent space of $\Phi_r(M)$ is obtained by
$$
T_{\gamma(r)}\Phi_r(M)={\rm Span}\{Y(r)|\ Y(t)\in\mathfrak{J}(M,\gamma(t))\}.
$$
Let $Y_i(t)$, $i=1,2,3$, be the $M$-Jacobi fields along the geodesic $\gamma(t)$
such that $Y_i(0)=E_i$ and $Y_i'(0)=-AE_i$. Then, by direct calculations, we obtain that
$$
\begin{aligned}
Y_i(t)&=\Big(\delta_{1i}\cos(\frac{t}{\sqrt{2}})-\sqrt{2}g(AE_1,E_i)\sin(\frac{t}{\sqrt{2}})\Big)E_1^t
+\Big(\delta_{2i}\cos(\frac{t}{\sqrt{2}})-\sqrt{2}g(AE_2,E_i)\sin(\frac{t}{\sqrt{2}})\Big)E_2^t\\
&+(\delta_{3i}-tg(AE_3,E_i))E_3^t.
\end{aligned}
$$
Taking $t=r$, we obtain that
\begin{equation}\label{tangentmap}
\left(
  \begin{array}{c}
    (\Phi_r)_*{E_1} \\
    (\Phi_r)_*{E_2} \\
    (\Phi_r)_*{E_3} \\
  \end{array}
\right)=
\left(
  \begin{array}{c}
    Y_1(r) \\
    Y_2(r) \\
    Y_3(r) \\
  \end{array}
\right)=
(B_{ij})\left(
  \begin{array}{c}
    E_1^r \\
    E_2^r \\
    E_3^r \\
  \end{array}
\right),
\end{equation}
where
$$
(B_{ij})=
\left(
\begin{array}{ccc}
\cos(\frac{r}{\sqrt{2}})-\sqrt{2}b_1\sin(\frac{r}{\sqrt{2}}) & -\sqrt{2}b_2\sin(\frac{r}{\sqrt{2}}) & 0 \\
-\sqrt{2}b_2\sin(\frac{r}{\sqrt{2}}) & \cos(\frac{r}{\sqrt{2}})
-\sqrt{2}b_4\sin(\frac{r}{\sqrt{2}}) & 0 \\
0 & 0 & 1 \\
\end{array}
\right).
$$
By a direct calculation, we have
$$
{\rm Det}(B_{ij})=\tfrac{1}{2}\Big(1-2b_2^2+2b_1b_4+(1+2b_2^2-2b_1b_4)\cos(\sqrt{2}r)
-\sqrt{2}(b_1+b_4)\sin(\sqrt{2}r)\Big).
$$

In the following, we assume that $\Phi_r(M)$ is a parallel hypersurface of $M$ at distance $r$.
Then, for any $(p,q)\in M$, the corresponding $\{E_1^r,E_2^r,E_3^r\}$ forms an orthonormal basis
of $T_{\Phi_r((p,q))}\Phi_r(M)$.
It is straightforward to check that $N^r=(N_1^r,N_2^r)$ defined by
$$
N_1^r=\cos(\tfrac{r}{\sqrt{2}})N_1-\tfrac{1}{\sqrt{2}}\sin(\tfrac{r}{\sqrt{2}})p, \ \
N_2^r=\cos(\tfrac{r}{\sqrt{2}})N_2-\tfrac{1}{\sqrt{2}}\sin(\tfrac{r}{\sqrt{2}})q, \ (p,q)\in M,
$$
is a unit normal vector field to the hypersurface $\Phi_r(M)$.
By \eqref{eqn:CC}, we know that $\Phi_r(M)$  also has constant product angle function $C=0$.
Next, we calculate the derivatives of $Y_i(t),i=1,2,3,$ with respect to $t$ and  obtain that
\begin{equation}\label{eqn:AY}
\left(
  \begin{array}{c}
    A_rY_1(r) \\
    A_rY_2(r) \\
    A_rY_3(r) \\
  \end{array}
\right)=
\left(
  \begin{array}{c}
    -Y_1'(r) \\
    -Y_2'(r) \\
    -Y_3'(r) \\
  \end{array}
\right)
=
(D_{ij})
\left(
  \begin{array}{c}
    E_1^r \\
    E_2^r \\
    E_3^r \\
  \end{array}
\right),
\end{equation}
where $A_r$ is the shape operator of hypersurface $\Phi_r(M)$ with respect to $N^r$ and
$$
(D_{ij})=
\left(
\begin{array}{ccc}
b_1\cos(\frac{r}{\sqrt{2}})+\frac{1}{\sqrt{2}}\sin(\frac{r}{\sqrt{2}}) &
b_2\cos(\frac{r}{\sqrt{2}}) & 0 \\
b_2\cos(\frac{r}{\sqrt{2}}) & b_4\cos(\frac{r}{\sqrt{2}})
+\frac{1}{\sqrt{2}}\sin(\frac{r}{\sqrt{2}}) & 0 \\
0 & 0 & 0 \\
\end{array}
\right).
$$
Finally, using \eqref{tangentmap} and \eqref{eqn:AY}, we obtain that
\begin{equation*}
	\left(
	\begin{array}{c}
		A_rE_1^r \\
		A_rE_2^r \\
		A_rE_3^r \\
	\end{array}
	\right)=
	(B_{ij})^{-1}(D_{ij})
	\left(
	\begin{array}{c}
		E_1^r \\
		E_2^r \\
		E_3^r \\
	\end{array}
	\right).
\end{equation*}
Then \eqref{eqn:Arr} follows immediately and we have completed the proof of the Proposition \ref{prop:4.3}.
%
\end{proof}

In the following, by using Proposition \ref{prop:4.3}, we establish the relationship between hypersurfaces with constant sectional curvature and minimal hypersurfaces with $C=0$, which is one of the key observations in the proof of Theorem \ref{thm:1.1}.

\begin{theorem}\label{thm:4.6}
Let $M$ be a hypersurface of $\mathbb{S}^2\times\mathbb{S}^2$ with constant sectional curvature $\kappa$.
Then, the parallel hypersurface $\Phi_{\frac{\pi}{2\sqrt{2}}}(M)$
of $M$ at distance $\frac{\pi}{2\sqrt{2}}$,
is a minimal hypersurface in $\mathbb{S}^2\times\mathbb{S}^2$ with constant $C=0$.
\end{theorem}
\begin{proof}
Since $M$  has  constant sectional curvature $\kappa$, by Proposition \ref{prop:4.5}, $M$ has constant $C=0$ and $\kappa=\tfrac{1}{2}$.
We choose the orthonormal frame fields $\{E_1,E_2,E_3\}$ defined by \eqref{eqn:2.5}
and assume that \eqref{eqn:2.7} holds.
It follows from the proof of Proposition \ref{prop:4.5} that $b_3=b_5=b_6=0$ and $b_1b_4-b_2^2=\frac{1}{2}$.

Notice that, for any $r\in \mathbb{R}$,
$\Phi_r(M)$ is a parallel hypersurface of $M$ if and only if
${\rm Rank}((\Phi_r)_*)=3$ on $M$. For a general hypersurface $M$,
$\Phi_r(M)$ may not always be a hypersurface. However, when $M$ has constant sectional curvature,
we will prove that ${\rm Rank}((\Phi_r)_*)=3$ holds on $M$ for any $r$.

In fact, by Proposition \ref{prop:4.3}, we know that, for any $(p,q)\in M$, the tangent map of $\Phi_r$ satisfies the expression \eqref{eqn:rank}. The determinant of $(B_{ij})$ is given by \eqref{eqn:det}.
Since $b_1b_4-b_2^2=\tfrac{1}{2}$, we can simplify ${\rm Det}(B_{ij})$ and obtain that
$${\rm Det}(B_{ij})=1-\frac{1}{\sqrt{2}}(b_1+b_4)\sin{(\sqrt{2}r)}.$$
If there exists some $r$ such that ${\rm Det}(B_{ij})=0$ in local, then $M$ has constant mean curvature,
which contradicts Remark \ref{rem:4.1aa}.
Since our concern is  local, we may assume that ${\rm Det}(B_{ij})\neq0$ for any $(p,q)\in M$
and any fixed $r$. It follows that  ${\rm Rank}((\Phi_r)_*)=3$ and  $\Phi_r(M)$ is a
parallel hypersurface of $M$ at distance $r$.

In the following, we still use $A_r$ to denote the shape operator of $\Phi_r(M)$.
By using Proposition \ref{prop:4.3} again, we have that,
for any $(p,q)\in M$, the expression of $A_r$ at $\Phi_r((p,q))$ is given by \eqref{eqn:Arr}.
Then, by using $b_1b_4-b_2^2=\frac{1}{2}$, we have
$$
\left(
  \begin{array}{c}
    A_rE_1^r \\
    A_rE_2^r \\
    A_rE_3^r \\
  \end{array}
\right)=
\left(
\begin{array}{ccc}
 \frac{-b_1+b_4-(b_1+b_4)\cos(\sqrt{2}r)}{-2+\sqrt{2}(b_1+b_4)\sin(\sqrt{2}r)} &
 \frac{-2b_2}{-2+\sqrt{2}(b_1+b_4)\sin(\sqrt{2}r)} & 0 \\
\frac{-2b_2}{-2+\sqrt{2}(b_1+b_4)\sin(\sqrt{2}r)} &
\frac{b_1-b_4-(b_1+b_4)\cos(\sqrt{2}r)}{-2+\sqrt{2}(b_1+b_4)\sin(\sqrt{2}r)} & 0 \\
0 & 0 & 0 \\
\end{array}
\right)
\left(
  \begin{array}{c}
    E_1^r \\
    E_2^r \\
    E_3^r \\
  \end{array}
\right).
$$
It follows that the mean curvature $H(r)$ of the parallel hypersurface $\Phi_r(M)$ is given by
$$
H(r)=\frac{2(b_1+b_4)\cos(\sqrt{2}r)}{6-3\sqrt{2}(b_1+b_4)\sin(\sqrt{2}r)}.
$$
Thus, if we take $r=\frac{\pi}{2\sqrt{2}}$, then $H(\frac{\pi}{2\sqrt{2}})=0$.
This means that parallel hypersurface $\Phi_{\frac{\pi}{2\sqrt{2}}}M$
is a minimal hypersurface in $\mathbb{S}^2\times\mathbb{S}^2$ with constant $C=0$.
\end{proof}

Finally, we apply Theorem \ref{thm:3.2aa}, Proposition \ref{prop:4.5}, Proposition \ref{prop:4.3} and Theorem \ref{thm:4.6}
to complete the proof of Theorem \ref{thm:1.1}.


%

%

\vskip 0.2 cm
\noindent{\bf Completion of the proof of Theorem \ref{thm:1.1}}.

Since $M$  has  constant sectional curvature $\kappa$, from Proposition \ref{prop:4.5}, we know that $M$ has constant $C=0$
and the sectional curvature $\kappa=\tfrac{1}{2}$.
According to Theorem \ref{thm:4.6}, $M$ is a parallel hypersurface of a minimal hypersurface
$\tilde{M}$ at distance $\frac{\pi}{2\sqrt{2}}$,
in which  $\tilde{M}$ also
has constant $C=0$ and such $\tilde{M}$ has been classified in Theorem \ref{thm:3.2aa}.
However, conversely, the parallel hypersurface of the hypersurface in
Theorem \ref{thm:3.2aa} may not always has constant sectional curvature.
In the following, we will check the three cases of Theorem \ref{thm:3.2aa},
and find out in which case the parallel hypersurface has constant sectional curvature.

{\bf Case (1)}: $\tilde{M}$ is an open part of $M_0$.

In this case, we point out that $M_t$, $t\in(-1,1)$, themselves are a family
of parallel hypersurfaces, and their focal submanifold is the diagonal surface
$\Sigma=\{(p,p)\in \mathbb{S}^2\times \mathbb{S}^2|\ p\in \mathbb{S}^2\}$.
It can be checked  directly that $M_t$ does not have constant
sectional curvature for any $t\in(-1,1)$.

Hence,  Case (1) can not occur.

{\bf Case (2)}: $\tilde{M}$ is an open part of $M_{\tilde{a},\tilde{b}}=\{(\tilde{p},\tilde{q})\in
\mathbb{S}^2\times \mathbb{S}^2|\ \langle \tilde{p},\tilde{a}\rangle+\langle \tilde{q},\tilde{b}\rangle=0\}$
for some $\tilde{a},\tilde{b}\in\mathbb{S}^2$.

In this case, up to an isometry of $\mathbb{S}^2\times\mathbb{S}^2$,
we assume that $\tilde{a}=(0,0,1), \tilde{b}=(0,0,-1)$.
Then, $\tilde{M}$ is locally given by the immersion
$\tilde{\Phi}:\mathbb{R}^3\rightarrow \mathbb{S}^2\times \mathbb{S}^2,$
$$
\tilde{\Phi}(t_1,t_2,t_3)=\big(\tilde{p}(t_1,t_2),\tilde{q}(t_1,t_3)\big)=\cos(\tfrac{t_1}{\sqrt{2}})
\big((\cos t_2,\sin t_2,0),(\cos t_3,\sin t_3,0)\big)
+\sin(\tfrac{t_1}{\sqrt{2}})(\tilde{a},-\tilde{b}),
$$
with unit normal field  given by
$$
N=(N_1,N_2)=-\tfrac{1}{\sqrt{2}}\sin(\tfrac{t_1}{\sqrt{2}})
\big((\cos t_2,\sin t_2,0),-(\cos t_3,\sin t_3,0)\big)
+\tfrac{1}{\sqrt{2}}\cos(\tfrac{t_1}{\sqrt{2}})(\tilde{a},\tilde{b}).
$$
We define the map
$\tilde{\Phi}_{\frac{\pi}{2\sqrt{2}}}:\tilde{M}\rightarrow \mathbb{S}^2\times \mathbb{S}^2$ by
$$
\begin{aligned}
\tilde{\Phi}_{\tfrac{\pi}{2\sqrt{2}}}((\tilde{p},\tilde{q}))
=&\big(\cos(\tfrac{\pi}{4})\tilde{p}(t_1,t_2)+\sqrt{2}\sin(\tfrac{\pi}{4})N_1,
\cos(\tfrac{\pi}{4})\tilde{q}(t_1,t_3)+\sqrt{2}\sin(\tfrac{\pi}{4})N_2\big)\\
=&\big(\tfrac{\cos(\tfrac{t_1}{\sqrt{2}})-\sin(\tfrac{t_1}{\sqrt{2}})}{\sqrt{2}}(\cos t_2,\sin t_2,0)
+\tfrac{\cos(\tfrac{t_1}{\sqrt{2}})+\sin(\tfrac{t_1}{\sqrt{2}})}{\sqrt{2}}\tilde{a},\\
&\tfrac{\cos(\tfrac{t_1}{\sqrt{2}})+\sin(\tfrac{t_1}{\sqrt{2}})}{\sqrt{2}}(\cos t_3,\sin t_3,0)
+\tfrac{\cos(\tfrac{t_1}{\sqrt{2}})-\sin(\tfrac{t_1}{\sqrt{2}})}{\sqrt{2}}\tilde{b}\big),~~ (\tilde{p},\tilde{q})\in\tilde{M}.
\end{aligned}
$$
It can be checked directly that $M=\tilde{\Phi}_{\frac{\pi}{2\sqrt{2}}}(\tilde{M})$ is a parallel hypersurface
of $\tilde{M}$ at distance $\frac{\pi}{2\sqrt{2}}$, and $M$ satisfies that
$\langle p,\tilde{a}\rangle^2+\langle q,\tilde{b}\rangle^2=1$
for any $(p,q)=\tilde{\Phi}_{\tfrac{\pi}{2\sqrt{2}}}((\tilde{p},\tilde{q}))\in M$.
Furthermore, $M$ has constant sectional curvature $\frac{1}{2}$ and $C=0$.

Hence, in Case (2),
$M$ is locally an open part of $\hat{M}_{a,b}$ for some $a,b\in\mathbb{S}^2$.

{\bf Case (3)}: $\tilde{M}$ is a minimal hypersurface described by Theorem \ref{thm:3.1}.

In this case, we still choose the orthonormal frame fields $\{E_1,E_2,E_3\}$  defined by \eqref{eqn:2.5} on $\tilde{M}$
 and assume that \eqref{eqn:2.7} holds, then $b_3=b_5=b_6=0$ and $H=\frac{b_1+b_4}{3}=0$ on $\tilde{M}$.

We adapt the notations in  Proposition \ref{prop:4.3}.
By Proposition \ref{prop:4.3}, we know that, for any $(p,q)\in \tilde{M}$, the tangent map of $\Phi_r$ satisfies the expression \eqref{eqn:rank}. The determinant of $(B_{ij})$ is given by \eqref{eqn:det}.
$H=0$ implies that $b_4=-b_1$, we can simplify ${\rm Det}(B_{ij})$ and obtain that
$${\rm Det}(B_{ij})=\frac{1}{2}\big(1-2b_1^2-2b_2^2+(1+2b_1^2+2b_2^2)\cos(\sqrt{2}r)\big).$$
If there exists some $r$ such that ${\rm Det}(B_{ij})=0$ in local,
then $\tilde{M}$ has constant scalar curvature, which contradicts Remark \ref{rem:3.1aa}.
Since our concern is local, we may assume that ${\rm Det}(B_{ij})\neq0$
for any $(p,q)\in \tilde{M}$ and any fixed $r$. It follows that
${\rm Rank}((\Phi_r)_*)=3$ and $\Phi_r(\tilde{M})$ is a parallel hypersurface of $\tilde{M}$ at distance $r$.

We still use  $A_r$ to denote  the shape operator of $\Phi_r(\tilde{M})$.
By applying Proposition \ref{prop:4.3} again, we have that,
for any $(p,q)\in \tilde{M}$, the expression of $A_r$ at $\Phi_r((p,q))$ is given by \eqref{eqn:Arr}.
Then, by using $b_4=-b_1$, we have
$$
\left(
  \begin{array}{c}
    A_rE_1^r \\
    A_rE_2^r \\
    A_rE_3^r \\
  \end{array}
\right)=
\left(
\begin{array}{ccc}
\frac{4b_1+\sqrt{2}(1+2b_1^2+2b_2^2)\sin(\sqrt{2}r)}
{2(1-2b_1^2-2b_2^2+(1+2b_1^2+2b_2^2)\cos(\sqrt{2}r))} &
\frac{2b_2}{1-2b_1^2-2b_2^2+(1+2b_1^2+2b_2^2)\cos(\sqrt{2}r)} & 0 \\
\frac{2b_2}{1-2b_1^2-2b_2^2+(1+2b_1^2+2b_2^2)\cos(\sqrt{2}r)} &
\frac{-4b_1+\sqrt{2}(1+2b_1^2+2b_2^2)\sin(\sqrt{2}r)}
{2(1-2b_1^2-2b_2^2+(1+2b_1^2+2b_2^2)\cos(\sqrt{2}r))} & 0 \\
0 & 0 & 0 \\
\end{array}
\right)
\left(
  \begin{array}{c}
    E_1^r \\
    E_2^r \\
    E_3^r \\
  \end{array}
\right).
$$
Armed with \eqref{eqn:2.8}, by direct calculations, we obtain that the Ricci curvature tensor of the hypersurface
$M=\Phi_{r}(\tilde{M})$ satisfies that

$$
\begin{aligned}
{\rm Ric}(E_1^r)&=\frac{-1+2b_1^2+2b_2^2}{(-1+2b_1^2+2b_2^2)-(1+2b_1^2+2b_2^2)\cos(\sqrt{2}r)}
E_1^r,\\
{\rm Ric}(E_2^r)&=\frac{-1+2b_1^2+2b_2^2}{(-1+2b_1^2+2b_2^2)-(1+2b_1^2+2b_2^2)\cos(\sqrt{2}r)}
E_2^r, \\
{\rm Ric}(E_3^r)&=E_3^r.
\end{aligned}
$$
If we take $r=\frac{\pi}{2\sqrt{2}}$, then
$
{\rm Ric}(E_1^r)=E_1^r,~{\rm Ric}(E_2^r)=E_2^r,~{\rm Ric}(E_3^r)=E_3^r.
$
Since  ${\rm dim}~M=3$,  by applying Schur Lemma, $M$ has constant sectional curvature $\tfrac{1}{2}$.

Hence, in  Case (3),
the  parallel hypersurface $\Phi_{\frac{\pi} {2\sqrt{2}}}(\tilde{M})$
of $\tilde{M}$ at distance $\frac{\pi}{2\sqrt{2}}$ has constant
sectional curvature $\tfrac{1}{2}$.
\qed

\vskip 0.2cm
\noindent{\bf Completion of the proof of Theorem \ref{thm:1.2}}.

Based on Theorem 1 in \cite{Ur}, by applying the classification results in
Theorem \ref{thm:3.2aa} and Theorem \ref{thm:1.1}, we complete the proof
of Theorem \ref{thm:1.2}.

\vskip 2cm



\vskip 10mm

\begin{flushleft}
Haizhong Li:\\
{\sc Department of Mathematical Sciences, Tsinghua University,\\
Beijing, 100084, P.R. China.}\\
E-mails: lihz@tsinghua.edu.cn.

\vskip 1mm

Luc Vrancken:\\
{\sc
Univ. Polytechnique Hauts-de-France,
CERAMATHS-Laboratoire de Math\'eriaux C\'{e}ramiques et de Math\'{e}matiques,
F-59313 Valenciennes, France;\\
INSA Hauts-de-France, CERAMATHS, F-59313 Valenciennes, France;\\
Department of Mathematics, KU Leuven, Celestijnenlaan 200B, Box 2400, BE-3001 Leuven, Belgium.}\\
E-mail: luc.vrancken@uphf.fr.

\vskip 1mm

Xianfeng Wang:\\
{\sc School of Mathematical Sciences and LPMC, Nankai University,\\
Tianjin 300071, P.R. China.}\\
E-mail: wangxianfeng@nankai.edu.cn.

\vskip 1mm

Zeke Yao:\\
{\sc Department of Mathematical Sciences, Tsinghua University,\\
Beijing, 100084, P.R. China.}\\
E-mails: yaozkleon@163.com.

\end{flushleft}

\end{document}